\documentclass[10pt,hidelinks]{article}
\usepackage[T1]{fontenc}
\usepackage{amsmath,amsfonts,amsthm,mathrsfs,amssymb}
\usepackage{cite}
\usepackage{multirow}
\usepackage{algorithmic,algorithm}
\usepackage{graphicx,float}
\usepackage{placeins}
\usepackage{color,bm}
\usepackage{xcolor}

\usepackage{indentfirst}
\usepackage{subfig}
\usepackage{tabularx}
\usepackage{tikz}
\usepackage[bookmarks=true]{hyperref}
\usepackage{url}

\graphicspath{{./fig/}} 
\numberwithin{equation}{section}
\newcommand{\R}{{\mathbb R}}

\newcommand{\be}{\begin{eqnarray}}
\newcommand{\ben}{\begin{eqnarray*}}
\newcommand{\en}{\end{eqnarray}}
\newcommand{\enn}{\end{eqnarray*}}

\newtheorem{theorem}{Theorem}[section]

\newtheorem{corollary}[theorem]{Corollary}
\newtheorem{definition}[theorem]{Definition}
\newtheorem{remark}[theorem]{Remark}

\definecolor{Taor1}{rgb}{0.000,0.000,0.000}
\definecolor{Taor25}{rgb}{1.000,0.000,0.000}

\begin{document}
\renewcommand{\theequation}{\arabic{section}.\arabic{equation}}
\begin{titlepage}
\title{From multi-layered problems to multiple two-layered problems: a novel frequency-time hybrid multiple-scattering integral equation solver}
\author{
Shuai Pan\thanks{School of Mathematical Sciences, University of Electronic Science and Technology of China, Chengdu, Sichuan 611731, China. Email:{\tt pans@std.uestc.edu.cn}},\;
Tao Yin\thanks{State Key Laboratory of Mathematical Sciences and Institute of Computational Mathematics and Scientific/Engineering Computing, Academy of Mathematics and Systems Science, Chinese Academy of Sciences, Beijing 100190, China. Email:{\tt yintao@lsec.cc.ac.cn}},\;
Lu Zhang\thanks{Division of Mathematical Sciences, School of Physical and Mathematical Sciences, Nanyang Technological University, 637371, Singapore. Email:{\tt luzhang@ntu.edu.sg}}
}
\date{}
\end{titlepage}
\maketitle
%
\begin{abstract}
This paper proposes a novel frequency-time hybrid multiple scattering (FTH-MS) integral equation solver for time-dependent wave equation problems in general multi-layered media, with remarkable scalability with respect to the number of layers $N$. In light of the finite speed of wave propagation, the new methodology provides an innovative multiple-scattering idea of re-modeling the original $N$-layered problem into a sequence of $N-1$ two-layered sub-problems, for which the main advantages lie in that (i) each sub-problem enjoys much simpler wave scattering properties compared with the complicated problem in a multi-layered medium, (ii) it enables to develop high-accuracy solver utilizing Fourier transform and frequency-domain boundary integral equation (BIE) method; and (iii) numerical evaluation of the sub-problems in each multiple scattering step can be parallelized. Both multiplicative- and additive-type strategies are developed and equivalence results, which indicate that the $M$-th order multiple scattering sums can provide equivalent representations of the solutions up to a certain time $T(M)$, are rigorously derived. Owing to the existed result of exponential convergence of the perfectly-matched-layer (PML) truncation for two-layered problem, all the sub-problems is numerically resolved by means of the FTH method based on the Fourier transform and the PML-BIE method whose numerical evaluation is addressed utilizing the Chebyshev-based rectangular-polar solver with high accuracy. Numerical examples are presented to validate the efficiency and accuracy of the proposed method.
\end{abstract}

\section{Introduction}




The accurate and efficient simulation of wave scattering in multi-layered media, in both the frequency and time domains, is a fundamental problem with wide-ranging applications in geophysics, material science, medical ultrasound, photonic devices, and many other fields. Understanding how acoustic or electromagnetic waves propagate, reflect, transmit, and interfere across multiple interfaces is essential for both accurate physical modeling and the design of advanced materials and devices. However, developing fast and high-accuracy numerical schemes for such kind of multi-layered scattering problems still remains significantly challenging since repeated reflections and transmissions between layers will create complex multiple-scattering effects, while phase accumulation can produce strong interference and resonance. 

For frequency-domain problems, classical volume-discretization techniques, such as the finite-difference method (FDM)~\cite{Chew3dpml1994,DURU2014757,OSKOOI2010687,DURU2025114268} and the finite-element method (FEM)~\cite{Jiang2022}, require the originally unbounded computational domain to be truncated by appropriate artificial boundary treatments, typically a transparent boundary condition (TBC) or a perfectly matched layer (PML). Despite their widespread success in scattering problems involving bounded obstacles~\cite{Lassas2001}, extending these techniques to general multi-layered problems remains nontrivial. In particular, because the interfaces between adjacent layers are themselves unbounded, the construction of an exact TBC based on the classical Dirichlet-to-Neumann map cannot be carried over directly from the bounded-obstacle setting and remains insufficiently understood. The theoretical justification of PML truncation presents a similar challenge: although exponential convergence has been established for two-layered media~\cite{chen2010,Chen2017PML}, a rigorous convergence theory for general $N$-layered configurations is still open. 

In contrast to volumetric discretization methods, boundary integral equation (BIE) method~\cite{Nedelec2001,ColtonKress2019Inversea} represents the wave fields in terms of layer potentials and reformulates the original scattering problem as a system of integral equations posed on material interfaces. They therefore offer several appealing advantages, including a reduction in spatial dimension and the automatic enforcement of the radiation condition at infinity. However, for layered-media problems, the method of using layered Green's function--which automatically enforces the transmission conditions across the unperturbed and unbounded planar interfaces--generally requires computation of challenging Fourier integrals containing highly-oscillatory integrands over infinite integration intervals. The alternative approach utilizing free-space fundamental solution--whose evaluation is considerably simpler and less expensive--however, results into integral equations imposed on the complete unbounded surface, making an additional truncation strategy necessary for numerical computation. Representative techniques include the windowed Green function (WGF) method~\cite{bruno2017windowed,Arancibia2022WindowedGF} and the PML-based BIE method~\cite{Lu2018Perfectly,Lu2023HighlyPML,BaoEtAl2024Highly}. Although the WGF method enjoys superalgebraic convergence, a relatively large window—and hence a large truncated interface—may be required to attain high accuracy, particularly at high frequencies or for complex multi-layered configurations. On the other hand, as discussed above, while PML truncation has proved highly effective for two-layered problems, its convergence and stability for general multi-layered media remain insufficiently understood.

Turning to the time-domain problems, the time-domain boundary integral equation (TDBIE) method based on retarded-potential representations of the wave fields has attracted much attention recently~\cite{Ha-Duong2003,Aimi2011,Barnett2020,Steinbach2022,Hoonhout2026}. A direct discretization of the TDBIE requires the accurate treatment of time-dependent integration regions determined by the intersection of the light cone with the entire scattering surface and therefore, the resulting numerical schemes are often complicated and present challenges concerning numerical stability~\cite{Barnett2020}. A particularly successful alternative is the convolution quadrature (CQ) method~\cite{Lubich1988} whose extension to the study of layered-medium problems in combination with the WGF method have been discussed in~\cite{labarca2019convolution}. The CQ method reduces the original time-domain problem to a sequence of frequency-domain problems with complex frequencies, followed by the numerical evaluation of an inverse $\mathcal{Z}$-transform~\cite{BanjaiRK2011}. This approach avoids the explicit discretization of retarded potentials and enables the reuse of well-established frequency-domain solvers. Nevertheless, the achievable accuracy is influenced by the numerical contour integration employed to approximate the inverse $\mathcal{Z}$-transform, as well as by the order and stability properties of the underlying time discretization.

Alternatively, Fourier-transform technique can also be employed to recast time-domain wave scattering problems as families of frequency-domain problems with real frequencies. For scattering by bounded obstacles, this approach is particularly effective under suitable geometric and non-trapping assumptions. Recently, a highly accurate frequency–time hybrid (FTH) method was introduced in~\cite{anderson2020high} for efficient long-time simulations by using a time-windowing and recentering strategy and high-order Fourier-transform algorithm. Each frequency-domain sub-problem can be effectively treated through BIE approach in the frequency domain over a controlled range of real frequencies and subsequently transformed back to the time domain. To extend the applicability of Fourier-based technique to a broader class of wave-propagation problems, a family of multiple-scattering frequency–time hybrid (FTH-MS) methods was subsequently developed in~\cite{bruno2024multiple,pan2025multi}. The central idea is to use a partition of unity to decompose a closed boundary, or an open surface exhibiting strong wave-trapping effects, into a collection of overlapping open patches. The original problem is then reformulated as a sequence of multiple-scattering interactions among these open patches. Since the individual patches do not enclose resonant cavities, the corresponding local scattering problems remain uniquely solvable at all real frequencies and the resulting formulation thus provides an elegant mechanism for avoiding the non-uniqueness and severe ill-conditioning that may arise at, or near, interior eigenfrequencies. Moreover, when the constituent patches are sufficiently simple and the poles of their associated resolvent operators remain well separated from the real axis~\cite{LafontaineEtAl2021Most}, the decomposition can substantially improve the conditioning of the resulting linear systems and reduce the number of GMRES iterations required for convergence. A closely related decomposition has also been applied to configurations involving multiple bounded obstacles in~\cite{pan2025multi}. In that setting, the global multiple-obstacle problem is recast as a coupled system of component-scattering problems, each involving only a single obstacle, while the interactions among different obstacles are accounted for through successive wave exchanges. Motivated by these developments, a natural question, illustrated in the box below, is whether the FTH-MS framework can be extended to wave scattering in general multi-layered media by decomposing the global $N$-layer problem into a sequence of coupled two-layered scattering sub-problems. Such a decomposition is particularly appealing because each constituent problem involves only a single interface separating two adjacent media, for which the underlying analytical structure and available numerical techniques are substantially better understood.
\[
\boxed{
\renewcommand{\arraystretch}{1.5}
\begin{array}{r@{\;}c@{\;}l}
N\text{-obstacle problem}
&
\overset{\text{\cite{pan2025multi}}}{\Longleftrightarrow}
&
\displaystyle\bigcup_{j=1}^N
\text{single-obstacle problems}
\\
&
\big\downarrow
&
\\
N\text{-layered problem}
&
\Longleftrightarrow
&
\displaystyle\bigcup_{j=1}^{N-1}
\text{two-layered problems}
\end{array}
}
\]

To realize this idea, we exploit the domain-of-influence property associated with the finite propagation speed of time-domain waves and propose a novel multiple-scattering algorithm for general $N$-layered media ($N\ge 3$). Rather than solving the wave equation directly in the full multi-layered configuration, the proposed method represents its solution by a multiple-scattering series generated through $N-1$ two-layered sub-problems, each associated with a single interface between two adjacent media. The global wave field is then reconstructed by successively accounting for the waves reflected and transmitted among the interfaces. For clarity, we first introduce both multiplicative- and additive-type multiple-scattering algorithms for the simplest nontrivial case $N=3$. Then extension of the additive formulation to general configuration with $N\ge 3$, as its structure is particularly well suited to parallel computation, is investigated. Both formulations recast the original multi-layered problem as a sequence of scattering interactions among the unbounded interfaces, yielding an expansion whose individual terms are obtained by solving two-layered sub-problems. By causality and finite-speed wave propagation, a truncated expansion can be proved to coincide with the solution of the original multi-layered wave equation over a finite time interval whose length increases with the truncation order and is determined by the wave speeds and the propagation distances between adjacent interfaces. Consequently, any prescribed finite simulation interval can, in principle, be covered by retaining a sufficiently large number of multiple-scattering terms. 

Building on this multiple-scattering decomposition, we next incorporate the FTH strategy~\cite{anderson2020high} by applying the Fourier transform to each two-layered sub-problem. This procedure converts the corresponding time-domain wave equation into a family of independent two-layered scattering problems with real frequencies and results into the expected FTH-MS solver, see Algorithm~\ref{alg:msN}. In this work, these frequency-domain problems are solved using the PML-BIE method which enjoys the effective truncation of single unbounded interface provided by the PML. Finally, the time-domain solution of each sub-problem is reconstructed through a high-order inverse Fourier transform, and the resulting component fields are assembled according to the proposed multiple-scattering algorithm. As demonstrated by the numerical experiments, this combination yields a modular computational framework that is well suited to high-order discretization and parallelization. Further discussions of the low-frequency regime and long-time simulations are also provided. In particular, we highlight the deterioration in the efficiency and accuracy of PML truncation at low frequencies, which poses a major obstacle, as shown in Example 3, to robust and high-accuracy long-time computation. A systematic treatment of this issue is left for future work.

The remainder of this paper is organized as follows. Section~\ref{sec:2} describes the mathematical formulations of the wave equation problem in a multi-layered medium and provides essential background on domain-of-influence property as well as the FTH-MS method corresponding to the problem with bounded obstacles. Both the ``multiplicative-type'' and ``additive-type''  multiple scattering methods to re-model the problem as combinations of a series of two-layered medium problems are developed in Section~\ref{sec:3} where the particular case $N=3$ and the extension to the general case $N\ge3$ are discussed in Sections~\ref{sec:3.1} and \ref{sec:3.1}, respectively. Based on the Fourier transform and PML-BIE method for solving the frequency-domain two-layered medium problems, Section~\ref{sec:4} proposes the FTH-MS integral equation solver whose accuracy and efficiency are validated in Section~\ref{sec:5} through several numerical experiments.

\section{Preliminaries}
\label{sec:2}

\subsection{Wave equation problem in a multi-layered medium}
\label{sec:2.1}

This work considers the problem of time-domain wave equation in a general $N$-layered medium. We only present the formulas in two dimensions and extension of the proposed methodology to the more challenging three-dimensional case in a straightforward manner. As shown in Figure~\ref{fig:layered-media}, we denote by $\Omega_j\subset\mathbb{R}^2$, $j = 1, \cdots, N$, the domain corresponding to each layered-medium separated by $N-1$ interfaces denoted by $\Gamma_l$, $l = 1, \dots, N-1$ each of which is a local perturbation of the flat surface. The finite propagation wave speed in each medium $\Omega_j$ is given by $c_j>0$, $j=1,\cdots,N$. The distance between $\Gamma_j$ and $\Gamma_{j+1}$ is denoted by $d_{j,j+1} = \mathrm{dist}(\Gamma_j, \Gamma_{j+1})$ for $j=1,\cdots,N-2$. Let $d_{min} = \min_{j=1,\cdots,N-2}\{d_{j,j+1}\}>0$. An application of the proposed algorithm to more complicated three-dimensional problem will be presented in Section~\ref{sec:5}.


Consider a smooth incident wave $u^{\text{inc}}$ which is assumed to arrives at the interfaces at a particular time $t_0>0$ such that both the resulted reflected fields and the transmitted fields are equal to zero at $t\le t_0$. The incident wave can be a point source $u^{\text{inc}}(\bm x,t)=u_{\text{point}}^{\text{inc}}(\bm x,t;\bm z_0)$ located at $\bm z_0\in\Omega_{j_0}$ or a plane wave $u^{\text{inc}}(\bm x,t)=u_{\text{plane}}^{\text{inc}}(\bm x,t;d^{\text{inc}})$ impinging on the first interface $\Gamma_1$ with incident direction $d^{\text{inc}}=(\cos\theta^{\text{inc}},\sin\theta^{\text{inc}})^\top, \theta^{\text{inc}}\in(-\pi/2,\pi/2)$. Then this work is devoted to studying the problem of transmitted and reflected fields induced by $u^{inc}$ satisfying the wave equations:
\begin{equation}
    \frac{\partial^2 u_j(\bm x,t)}{\partial t^2} - c_j^2 \Delta u_j(\bm x,t) = 0, \quad (\bm x,t) \in \Omega_j \times \mathbb{R}^+,\quad j=1,\cdots,N,
    \label{eq-transmisson-N}
\end{equation}
with zero initial conditions $u_j|_{t=0}=\frac{\partial u_j}{\partial t}|_{t=0}=0$ in $\Omega_j$, $j=1,\cdots,N$. Here, $\mathbb{R}^+:=\{t\in \mathbb{R}: t>0\}$.
On each interface $\Gamma_j$, $j=1,\cdots,N$ between the layers $\Omega_j$ and $\Omega_{j+1}$, the fields are coupled through the transmission conditions
\begin{equation}
    \begin{cases}
        u_{j+1}(\bm x,t) - u_j(\bm x,t) = f_j(\bm x,t), \\
        \rho_{j+1}^{-1} \frac{\partial u_{j+1}(\bm x,t)}{\partial \bm n} - \rho_j^{-1} \frac{\partial u_j(\bm x,t)}{\partial \bm n} = g_j(\bm x,t),
    \end{cases}
    \quad \text{on } \Gamma_j \times \mathbb{R}^+,
    \label{eq-transmisson-condition-N}
\end{equation}
where the constant $\rho_j>0$ denotes the density of the medium in $\Omega_j$ and $\bm n$ is the unit normal directing from $\Omega_{j+1}$ to $\Omega_j$ on $\Gamma_j$, see Figure~\ref{fig:layered-media}. The data $f_j$ and $g_j$ in the transmission conditions is determined by the incident field $u^{\text{inc}}$. In particular, for the case of plane wave incidence $u_{\text{plane}}^{\text{inc}}$:
\begin{equation}
    f_j = \delta_{j,1} u_{\text{plane}}^{\text{inc}}, \quad g_j = \delta_{j,1} \rho_1^{-1} \frac{\partial u_{\text{plane}}^{\text{inc}}}{\partial n},
    \label{eq:transmission-condition-plane-wave}
\end{equation}
where $\delta_{i,j}$ is the Kronecker delta function. For the incidence of a point source $u_{\text{point}}^{\text{inc}}$:
\begin{equation}
    f_j = (\delta_{j,{j_0}} - \delta_{j+1,{j_0}}) u_{\text{point}}^{\text{inc}}, \quad g_j = (\rho_j^{-1}\delta_{j,{j_0}} - \rho_{j+1}^{-1}\delta_{j+1,{j_0}})  \frac{\partial u_{\text{point}}^{\text{inc}}}{\partial n}.
    \label{eq:transmission-condition-point-source}
\end{equation}

\begin{figure}
    \centering
    \includegraphics[width=0.4\linewidth]{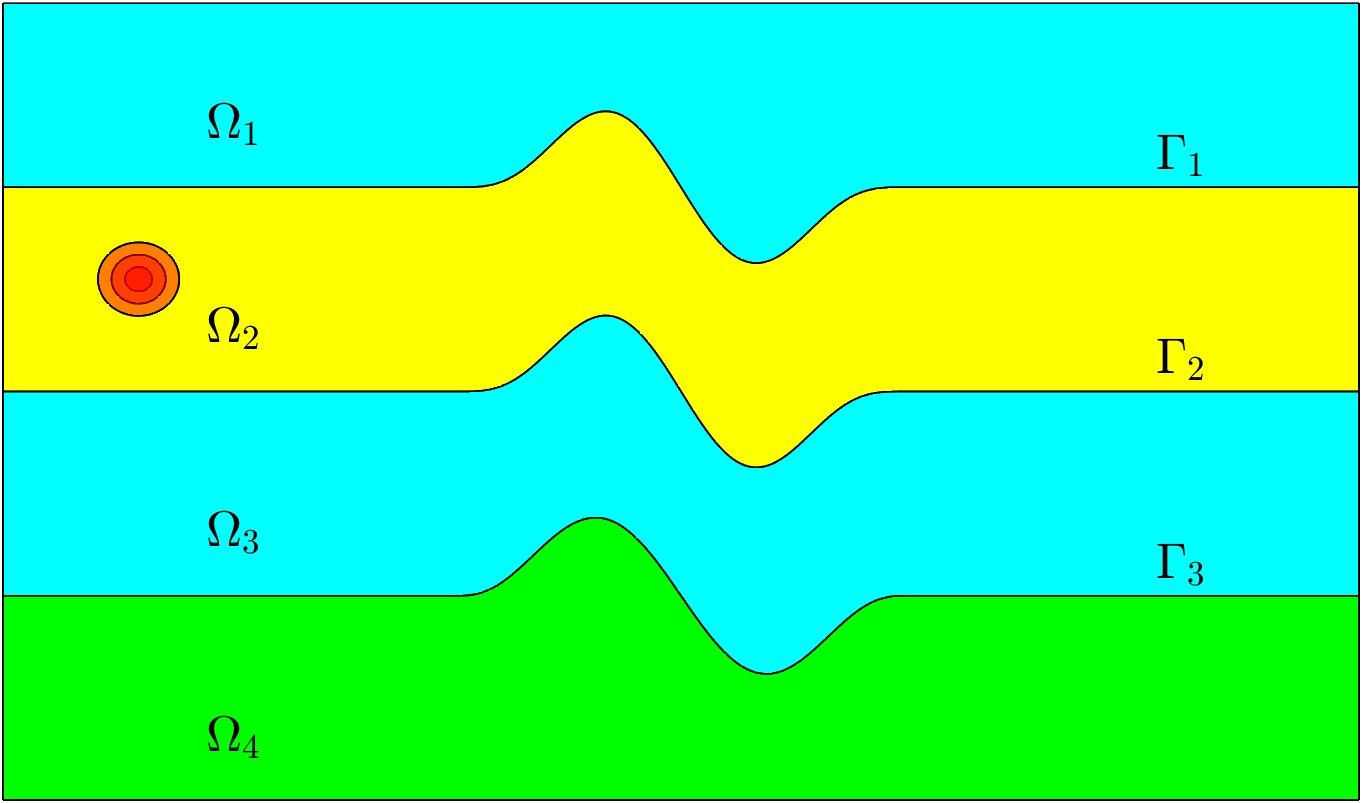}
    \includegraphics[width=0.4\linewidth]{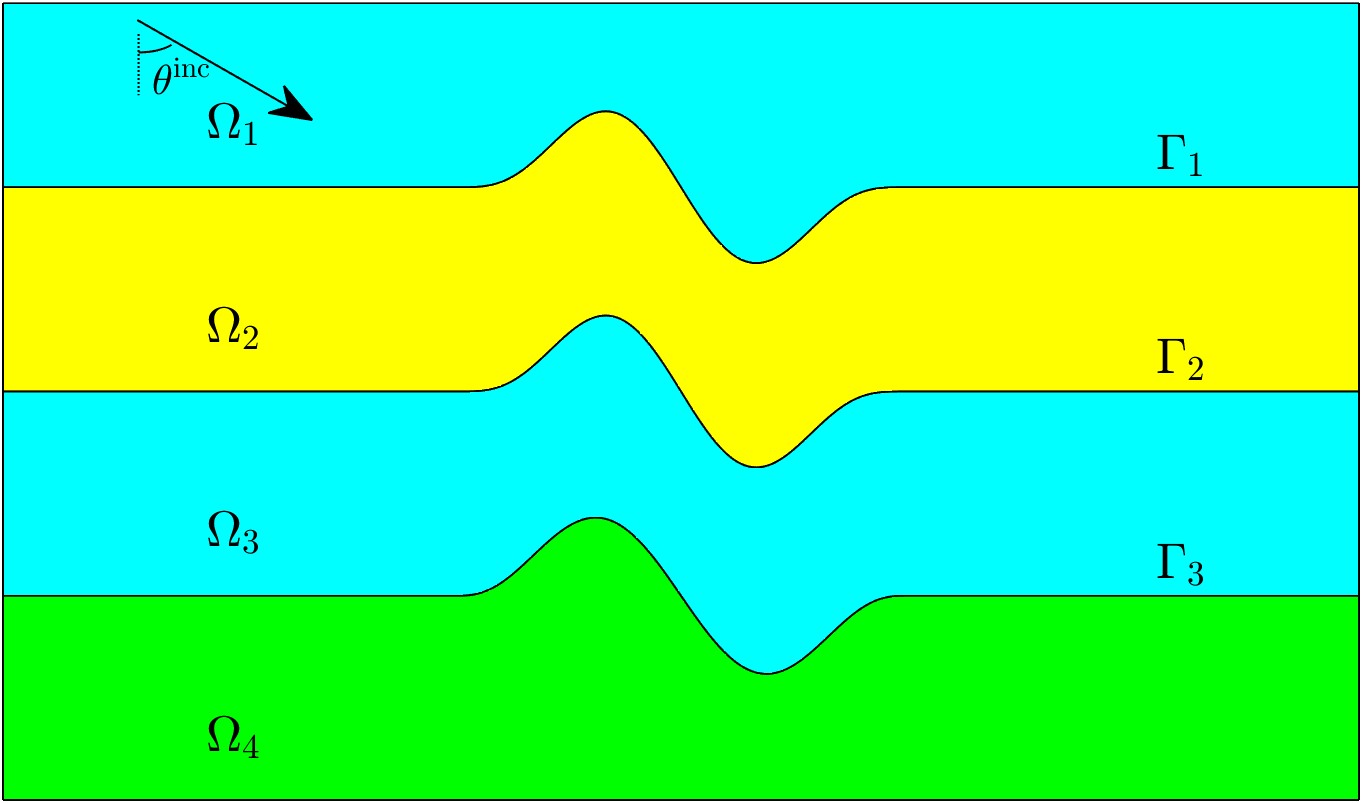}
    \caption{Layered media($N=4$) with locally perturbed interfaces. Left: point source incidence, Right: plane wave incidence.}
    \label{fig:layered-media}
\end{figure}

For the convenience of the discussion of our multiple scattering algorithm, we are also interested in the two-layered medium problem due to the scattering by each interface $\Gamma_j$, by denoting $\widetilde\Omega_j$ and $\widetilde\Omega_{j+1}$ the infinite domain above and below the interface $\Gamma_j$, respectively, as follows:
\begin{equation}
\begin{cases}
\frac{\partial^2 \widetilde u_j^+(\bm x,t)}{\partial t^2}-c_j^2\Delta \widetilde u_j^+(\bm x,t)=0,                                                       & (\bm x,t)\in \widetilde\Omega_{j}\times\mathbb{R}^+,\\
\frac{\partial^2 \widetilde u_{j+1}^-(\bm x,t)}{\partial t^2}-c_{j+1}^2\Delta \widetilde u_{j+1}^-(\bm x,t)=0,                                                       & (\bm x,t)\in \widetilde\Omega_{j+1}\times\mathbb{R}^+,\\
\widetilde u_{j+1}^-(\bm x,t)-\widetilde u_{j}^+(\bm x,t)=\widetilde f_j(\bm x,t), & (\bm x,t)\in \Gamma_j\times\mathbb{R}^+,   \\
\rho_{j+1}^{-1}\frac{\partial \widetilde u_{j+1}^-(\bm x,t)}{\partial \bm n}-\rho_{j}^{-1}\frac{\partial \widetilde u_{j}^+(\bm x,t)}{\partial \bm n}=\widetilde g_j(\bm x,t) & (\bm x,t)\in \Gamma_j\times\mathbb{R}^+,\\
\widetilde u_j^+(\bm x,t)|_{t=0}=\frac{\partial \widetilde u_j^+(\bm x,t)}{\partial t}|_{t=0}=0, & \bm x\in\widetilde\Omega_j,\\
\widetilde u_{j+1}^-(\bm x,t)|_{t=0}=\frac{\partial \widetilde u_{j+1}^-(\bm x,t)}{\partial t}|_{t=0}=0, & \bm x\in\widetilde\Omega_{j+1},
\end{cases}
\label{eq:two-layer0}
\end{equation}
with boundary data $\widetilde f_j, \widetilde g_j$ to be specified.

By the classical energy method (see for example~\cite{EvansPDE,DURU2014757}), the uniqueness results of the transmission problems~\eqref{eq-transmisson-N}--\eqref{eq-transmisson-condition-N} and (\ref{eq:two-layer0}) are concluded in the following theorem.
\begin{theorem}\label{th:transmission-well-posedness}
Let the interfaces $\Gamma_j, j=1,\cdots,N$ be Lipschitz. Then each of the transmission problems \eqref{eq-transmisson-N}--\eqref{eq-transmisson-condition-N} and \eqref{eq:two-layer0} has at most one solution.
\end{theorem} 

The aim of this work is to proposing a novel Fourier transform based frequency/time hybrid integral
equation solver for the time-domain multi-layered medium problem~\eqref{eq-transmisson-N}--\eqref{eq-transmisson-condition-N}, however, due to the possible existence of trapping modes in $\Omega_j, j=2,\cdots,N-1$~\cite{Nedelec_Waveguide}, constructing frequency-robust integral equation solvers for the frequency-domain multi-layered medium problems within a broadband frequency region, resulting from a direct application of the Fourier transform to the problem~\eqref{eq-transmisson-N}--\eqref{eq-transmisson-condition-N}, remains significantly challenging. In light of the idea of FTH-MS algorithm developed in \cite{pan2025multi} for solving the wave equation problem with multiple bounded obstacles, as shown in Section~\ref{sec:3}, the original time-domain multi-layered medium problem~\eqref{eq-transmisson-N}--\eqref{eq-transmisson-condition-N} can be reformulated as a multiple-scattering series of multiple two-layered medium problems. To end of this section, we next describe the necessary domain-of-influence property related to the layered-medium problems as well as the idea of FTH and FTH-MS algorithm.

\subsection{Domain-of-influence property}
\label{sec:2.2}

The Huygens-like domain-of-influence property satisfied by the wave equation problems plays a crucial role to show the efficiency of the multiple scattering algorithm. 
For the exterior problem of the wave equation with a bounded obstacle, the classical domain-of-influence~\cite{Sayas2016RetardedPA} states the relation between the compact support of the scattered field and the propagation speed.

\begin{theorem}[Domain-of-influence for the bounded-obstacle problem]
\label{DoI-bound}
Let $D\subset\mathbb{R}^2$ be a bounded obstacle with Lipschitz boundary $\Gamma_D$ and denote by $D^e=\mathbb{\R}^2\backslash\overline{D}$ the exterior region. Denote by $c>0$ the wave speed. Let $u_D$ be the unique solution to the wave equation problem with Dirichlet boundary condition:
\begin{equation}
\begin{cases}
\frac{\partial^2 u_D(\bm x,t)}{\partial t^2} - c^2 \Delta u_D(\bm x,t) = 0, & (\bm x,t) \in D^e \times \mathbb{R}^+,\cr
u_D(\bm x,t)=-u^{inc}(\bm x,t), &(\bm x,t) \in \Gamma_D \times \mathbb{R}^+,\cr
u_D(\bm x,t)|_{t=0}= \frac{\partial u_D(\bm x,t)}{\partial t}|_{t=0}=0, & \bm x\in D^e
\end{cases}
\label{eq-bound}
\end{equation}
where $u^{\text{inc}}$ is a causal incident field such that $u^{\text{inc}}(\bm x,t)=0$ for all $\bm x\in\Gamma_D$ and $t\le t_0$ with some $t_0\ge 0$. Then for any $\bm x\in D^e$, $u_D(\bm x,t)=0$ for all $0\le t\le t_0+c^{-1}\text{dist}(\bm x,\Gamma_D)$.
\end{theorem}

As stated in the following corollary, this fundamental domain-of-influence property can be straightly generalized to the wave equation problem in a two-layered medium (\ref{eq:two-layer0}).

\begin{corollary}[Domain-of-influence for the two-layered medium problem]
\label{DoI-twolayer}
For each $j=1,\cdots,N-1$, let $\widetilde u_j^+, \widetilde u_{j+1}^-$ be the unique solutions to the problem \eqref{eq:two-layer0} where the boundary data admits $\widetilde{f_j}=\widetilde{g_j}=0$ on $\Gamma_j$ for all $t\le t_0$ with some $t_0\ge 0$. Then for any $\bm x\in\Omega_j$, $\widetilde u_j^+(\bm x,t)=0$ for all $0\le t\le t_0+c_j^{-1}\text{dist}(\bm x,\Gamma_j)$ and for any $\bm x\in\Omega_{j+1}$, $\widetilde u_{j+1}^-(\bm x,t)=0$ for all $0\le t\le t_0+c_{j+1}^{-1}\text{dist}(\bm x,\Gamma_j)$.
\end{corollary}

The Huygens-like domain-of-influence property for the general $N$-layered medium problem \eqref{eq-transmisson-N}--\eqref{eq-transmisson-condition-N} is given in the following theorem. The result is shown for the case of point source only, but analogous result also holds for the plane incidence case. 

\begin{theorem}[Domain-of-influence for the general $N$-layered medium problem]
\label{DoI-Nlayer}
Let $u_j, j=1,\cdots,N$ be the unique solution to the problem \eqref{eq-transmisson-N}--\eqref{eq-transmisson-condition-N} with a causal point source located at $\bm z_0\in \Omega_{i_0}$. For any $\bm x\in \Omega_j$, we have $u_j(\bm x,t)=0$ for all $t\in [0,\eta(\bm x)+t_0]$, where $\eta(\bm x)$ is given by 
\begin{equation}
\eta(\bm x)=\begin{cases}
\sum_{k=j+1}^{i_0-1}\frac{d_{k-1,k}}{c_k}+\frac{\text{dist}(\bm x,\Gamma_j)}{c_j},     & j<i_0, \\
\frac{\min\{\text{dist}(\bm x,\Gamma_{j-1}),\text{dist}(\bm x,\Gamma_j)\}}{c_j},         & j=i_0, \\
\sum_{k=i_0+1}^{j-1}\frac{d_{k-1,k}}{c_k}+\frac{\text{dist}(\bm x,\Gamma_{j-1})}{c_j}, & j>i_0.
\end{cases}
\end{equation}
\end{theorem}
\begin{proof}
It can be shown from the energy method that \cite[Section~7.2.4]{EvansPDE} that the wave propagates with finite speed in the layered media. For any $\bm x\in \Omega_j$, let $\tau_{\bm x,j}$ be the arriving time of the wave front due to the incidence of a causal point source located at $\bm z_0\in \Omega_{i_0}$. Then when $j=i_0$, the solution $u_j(\bm x)$ is generated from the scattering by both the upper and lower interfaces of the source layer itself which yields
\begin{equation*}
\tau_{\bm x,j}> \frac{\min(\text{dist}(\bm x,\Gamma_{j-1}),\text{dist}(\bm x,\Gamma_j))}{c_j}.
\end{equation*}
For the case $j<i_0$, the path of the wave traverses upwards through the layers $\Omega_{i_0-1},\Omega_{i_0-2},\cdots, \Omega_{j+1}$, before finally entering $\Omega_j$ to reach $\bm x$. Hence,
\begin{equation*}
\tau_{\bm x,j}> \sum_{k=j+1}^{i_0-1}\frac{d_{k-1,k}}{c_k}+\frac{\text{dist}(\bm x,\Gamma_j)}{c_j}.
\end{equation*}
The result for $j>i_0$ follows analogously and thus, completes the proof.
\end{proof}

\subsection{FTH-MS solver for problem with multiple bounded obstacles}
\label{sec:2.3}

Let $D$ in the problem \eqref{eq-bound} be a set of $N$ well-separated bounded obstacles, i.e., $D=\cup_{j=1}^N D_j$ such that each $D_j\subset\R^2$ is a closed bounded obstacle with Lipschitz boundary $\Gamma_{D_j}$ and $\delta_{min}=\min_{1\leq j\leq N}\{\mathrm{dist}(\Gamma_{D_j},\Gamma_D\setminus\Gamma_{D_j})\}$ with $\Gamma_D=\cup_{j=1}^N \Gamma_{D_j}$. Then the multiple scattering algorithm allows to represent the solution $u_D$ to the problem \eqref{eq-bound} as a finite sum of wave equation solutions $v_{j,m}$
\begin{equation}
u_D^M(\bm x,t)=\sum_{m=1}^{M}\sum_{j=1}^{N}v_{j,m}(\bm x,t),\quad \bm x\in D^e,\; t\ge 0.
\label{eq:bound-sum}
\end{equation}
in which each $v_{j,m}$ satisfies a single-obstacle wave equation problem:
\begin{equation}
\begin{cases}
\frac{\partial^2 v_{j,m}(\bm x,t)}{\partial t^2}-c^2\Delta v_{j,m}(\bm x,t)=0, & (\bm x,t)\in D_j^e\times\R^+, \\
v_{j,m}(\bm x,t)= f_{j,m}(\bm x,t),                        & (\bm x,t)\in\Gamma_{D_j}\times\R^+,\\
v_{j,m}(\bm x,t)|_{t=0}= \frac{\partial v_{j,m}(\bm x,t)}{\partial t}|_{t=0}=0, & \bm x\in D_j^e.
\end{cases}
\label{eq:waveeqn-closedj}
\end{equation}
Here, $D_j^e=\R^2\backslash\overline{D_j}$. Then boundary data $f_{j,m}$, $j=1,\cdots, N$, $m=1,\cdots,M$ is determined through the following recursive process:
\begin{equation}
f_{j,1}(\bm x,t)=-u^{inc}(\bm x,t),\quad \bm x\in \Gamma_j,\quad t\in \R^+,
\label{eq:BDj1}
\end{equation}
and for $m\ge 1$,
\begin{equation}
f_{j,m+1}(\bm x,t)=-\sum_{k=1,k\neq j}^{N}v_{k,m}(\bm x,t),\quad \bm x\in \Gamma_j,\quad t\in \R^+.
\label{eq:BDjm}
\end{equation}
Utilizing the domain-of-influence property for the bounded-obstacle problem (Theorem~\ref{DoI-bound}), it can be shown that (see for example \cite[Theorem~3.5]{pan2025multi}) that $u_D^M(\bm x,t)=u_D(\bm x,t)$ for all $\bm x\in D^e$ and $0\le t\le c^{-1}M\delta_{min}$.

Based on the re-modeling (\ref{eq:bound-sum}), now it allows to apply the FTH solver developed in \cite{anderson2020high} to solve the wave equation problem (\ref{eq:waveeqn-closedj}) through the Fourier
transform. For any $f,G\in L^2(\mathbb{R})$, the corresponding Fourier transform $F$ and inverse Fourier transform $g$ are defined as
\begin{equation*}
F(\omega)=\mathbb{F}(f):=\int_{-\infty}^{\infty}f(t)e^{i\omega t}dt,\quad g(t)=\mathbb{F}^{-1}(G):=\frac{1}{2\pi}\int_{-\infty}^{\infty}G(\omega)e^{-i\omega t}d\omega,\quad \omega,t\in\R.
\end{equation*}
Let $V_{j,m}=\mathbb{F}(v_{j,m})$ and $F_{j,m}=\mathbb{F}(f_{j,m})$, then the problems \eqref{eq:waveeqn-closedj} can be transformed into frequency-domain problems of Helmholtz equation with wave number $\kappa(\omega)=\frac{\omega}{c}$:
\begin{equation}
    \begin{cases}
        \Delta V_{j,m}(\bm x,\omega)+\kappa^2 V_{j,m}(\bm x,\omega)=0,\quad \bm x\in D_j^e, \\
        V_{j,m}(\bm x,\omega)=F_{j,m}(\bm x,\omega),\quad \bm x\in \Gamma_{D_j},
    \end{cases}
    \label{eq:freq-domain-bounded}
\end{equation}
together with the Sommerfeld radiation condition. In particular, the problem (\ref{eq:freq-domain-bounded}) can be efficiently solved by advanced boundary integral equation solvers with high accuracy based on the methodology of using combined field integral equations which can be shown to be uniquely solvable for all $\omega>0$~\cite{ColtonKress2019Inversea}. The solutions $V_{j,m}(\bm x,\omega)$ for $\omega<0$ can be generated through taking the conjugation $\overline{V_{j,m}(\bm x,-\omega)}$.

\section{Multiple-scattering strategy}
\label{sec:3}

Note that the well-posedness of the corresponding frequency-domain problems of Helmholtz equation in a general multi-layered medium still remains open due to the possible existence of trapped waves and as aforementioned, developing frequency-robust integral equation solvers for those problems also remains challenging. In this section, instead of solving the wave equation problem \eqref{eq-transmisson-N}--\eqref{eq-transmisson-condition-N} in a general $N$-layered medium ($N\ge 3$) directly, we propose a novel multiple scattering algorithm providing an equivalent representation of the solution as a multiple scattering series of $N-1$ two-layered medium subproblems.  
For $m\in\mathbb{N}$ and $j=1,\cdots,N-1$, we denote by $\widetilde{u}_{j,m}^+, \widetilde{u}_{j+1,m}^-$ the solution to the problem (\ref{eq:two-layer0}) with boundary data $\widetilde{f}_{j,m}, \widetilde{g}_{j,m}$ which will be determined through certain recursive process, i.e.,
\begin{equation}
\begin{cases}
\frac{\partial^2 \widetilde u_{j,m}^+(\bm x,t)}{\partial t^2}-c_j^2\Delta \widetilde u_{j,m}^+(\bm x,t)=0,                                                       & (\bm x,t)\in \widetilde\Omega_{j}\times\mathbb{R}^+,\\
\frac{\partial^2 \widetilde u_{j+1,m}^-(\bm x,t)}{\partial t^2}-c_{j+1}^2\Delta \widetilde u_{j+1,m}^-(\bm x,t)=0,                                                       & (\bm x,t)\in \widetilde\Omega_{j+1}\times\mathbb{R}^+,\\
\widetilde u_{j+1,m}^-(\bm x,t)-\widetilde u_{j,m}^+(\bm x,t)=\widetilde f_{j,m}(\bm x,t), & (\bm x,t)\in \Gamma_j\times\mathbb{R}^+,   \\
\rho_{j+1}^{-1}\frac{\partial \widetilde u_{j+1,m}^-(\bm x,t)}{\partial \bm n}-\rho_{j}^{-1}\frac{\partial \widetilde u_{j,m}^+(\bm x,t)}{\partial \bm n}=\widetilde g_{j,m}(\bm x,t) & (\bm x,t)\in \Gamma_j\times\mathbb{R}^+,\\
\widetilde u_{j,m}^+(\bm x,t)|_{t=0}=\frac{\partial \widetilde u_{j,m}^+(\bm x,t)}{\partial t}|_{t=0}=0, & \bm x\in\widetilde\Omega_j,\\
\widetilde u_{j+1,m}^-(\bm x,t)|_{t=0}=\frac{\partial \widetilde u_{j+1,m}^-(\bm x,t)}{\partial t}|_{t=0}=0, & \bm x\in\widetilde\Omega_{j+1}.
\end{cases}
\label{eq:two-layer1}
\end{equation}

\subsection{Multiple scattering: $N=3$}
\label{sec:3.1}

For convenience, we first discuss our multiple scattering idea for the case $N=3$ in this subsection and the extension of the algorithm to the general case $N\ge 3$ will be presented in Section~\ref{sec:3.2}.

We first construct a ``multiplicative-type'' multiple scattering in which the wave propagation in the three-layered medium can be viewed as a ``ping-pong'' scattering between the two interfaces. As a result, time-domain solution of the original problem \eqref{eq-transmisson-N}--\eqref{eq-transmisson-condition-N} can be reformulated as the sum of ``ping-pong'' wave-equation solutions produced under multiple scattering. Here, we call
\begin{equation*}
j'=j'(m):=2-\mathrm{mod}(m,2),\quad m\in\mathbb{N},
\end{equation*}
and inductively define the boundary-condition functions $\widetilde{f}_{j',m}, \widetilde{g}_{j',m}$ and associated wave-equation solutions
$\widetilde u_{j',m}^+, \widetilde u_{j'+1,m}^-$, $m\in\mathbb N$ as follows.

\begin{definition}\label{def:ms3-bound-con-mul}
For $m\in\mathbb{N}$, $N=3$, let $\widetilde u_{j',m}^+, \widetilde u_{j'+1,m}^-$ be the solutions to the problem (\ref{eq:two-layer1}) where $\widetilde{f}_{j',m}, \widetilde{g}_{j',m}$ are defined inductively via the relations
\begin{align}
(\widetilde{f}_{1,1},\widetilde{g}_{1,1})& =(f_1,g_1), \quad m=1,\\
(\widetilde{f}_{2,2},\widetilde{g}_{2,2})& =\left(f_2+\widetilde u_{2,1}^-,g_2+ \rho_{2}^{-1}\frac{\partial \widetilde u_{2,1}^-}{\partial\bm n} \right),\quad m=2,
\end{align}
for $m\ge 3$ and $m$ is odd, 
\begin{align}
(\widetilde{f}_{1,m},\widetilde{g}_{1,m})& ={\left( -\widetilde u_{2,m-1}^+,  -\rho_{2}^{-1}\frac{\partial \widetilde u_{2,m-1}^+}{\partial\bm n}\right)},
\end{align}
and, for $m\ge 3$ and $m$ is even,
\begin{align}
(\widetilde{f}_{2,m},\widetilde{g}_{2,m})& ={\left( \widetilde u_{2,m-1}^-,  \rho_{2}^{-1}\frac{\partial \widetilde u_{2,m-1}^-}{\partial\bm n}\right)}.
\end{align}
\end{definition}

Following the evaluation steps
\begin{align*}
(\widetilde u_{1,1}^+, \widetilde u_{2,1}^-) \rightarrow (\widetilde u_{2,2}^+, \widetilde u_{3,2}^-) \rightarrow (\widetilde u_{1,3}^+, \widetilde u_{2,3}^-) \rightarrow (\widetilde u_{2,4}^+, \widetilde u_{3,4}^-) \rightarrow ... \rightarrow (\widetilde u_{j'(M),M}^+, \widetilde u_{j'(M)+1,M}^-),
\end{align*}
now we can define an $M$-th ($M\in\mathbb{N}$) order sum in terms of the solutions
$\widetilde u_{j',m}^+, \widetilde u_{j'+1,m}^-$, $m\in\mathbb N$ as
\begin{equation}
\begin{aligned}
\widehat u_{1}^M & = \sum_{n=1}^{\lceil M/2 \rceil}\widetilde u_{1,2n-1}^+,\quad \widehat u_{2}^M & = \sum_{n=1}^{\lceil M/2 \rceil} \widetilde u_{2,2n-1}^- + \sum_{n=1}^{\lfloor M/2 \rfloor} \widetilde u_{2,2n}^+, \quad \widehat u_{3}^M & = \sum_{n=1}^{\lfloor M/2 \rfloor} \widetilde u_{3,2n}^-.
\end{aligned}
\label{eq:threelayer-mul-ms}
\end{equation}
Then the domain-of-influence property induces the following equivalence result.

\begin{theorem}\label{th:ms3-mul}
Assume that $f_j,g_j=0$ for $t\le t_0, j=1,2$. Let $ M \in \mathbb{N}$ with $ M \geq 2 $, and define $T(M)=t_0+c_2^{-1}d_{1,2}M$. Then we have $ u_j(\bm x, t) = \widehat u_j^M(\bm x, t) $ for all $(\bm x, t)\in \Omega_j \times [0, T(M-1)]$, $j=1,2,3$.
\end{theorem}
\begin{proof}
By Theorem \ref{th:transmission-well-posedness}, it suffices to show that the multiple scattering series \eqref{eq:threelayer-mul-ms} satisfies
\begin{align}
    &\widehat u_{j+1}^M(\bm x, t)-\widehat u_{j}^M(\bm x, t)=f_j(\bm x, t),\quad  & (\bm x,t)\in \Gamma_j\times [0,T(M-1)], \label{eq:sub-trans1}\\
    &\rho_{j+1}^{-1}\frac{\partial \widehat u_{j+1}^M(\bm x,t)}{\partial \bm n}-\rho_{j}^{-1}\frac{\partial \widehat u_{j}^M(\bm x,t)}{\partial \bm n}=g_{j}(\bm x,t), & (\bm x,t)\in \Gamma_j\times [0,T(M-1)],\label{eq:sub-trans2}
\end{align}
for $j=1,2$.
To do this, we first show that 
\begin{equation}
    \widetilde{f}_{j'(M),M}(\bm x, t) =\widetilde{g}_{j'(M),M}(\bm x, t)= 0, \quad (\bm x, t) \in \Gamma_{j'(M)} \times [0, T(M-2)].
    \label{eq:casulity-trans-con}
\end{equation}
By the Definition \ref{def:ms3-bound-con-mul}, $\widetilde{f}_{1,1},\widetilde{g}_{1,1}$ are casual functions satisfying $\widetilde{f}_{1,1}=\widetilde{g}_{1,1}=0$ for $t\le t_0$, thus $\widetilde u_{2,1}^-(\bm x, t)=\frac{\partial \widetilde u_{2,1}^-(\bm x, t)}{\partial\bm n}=0$ for $(\bm x, t)\in \Gamma_2\times [0,T(1)]$ by Corollary~\ref{DoI-twolayer}. Thus $\widetilde f_{2,2}=f_2+\widetilde u_{2,1}^-$ and $\widetilde g_{2,2}=g_2+ \rho_{2}^{-1}\frac{\partial \widetilde u_{2,1}^-}{\partial\bm n}$ are also casual functions such that (\ref{eq:casulity-trans-con}) holds for $M=2$. For $M> 2$, We prove it by induction on $ M $.
Assume that \eqref{eq:casulity-trans-con} holds for any $M\in \mathbb N$ with $ 2\le M \le L,L\ge 2$, by Corollary~\ref{DoI-twolayer}, it follows that
\begin{align*}
     \widetilde u_{2,L}^-(\bm x, t)&=\frac{\partial \widetilde u_{2,L}^-(\bm x, t)}{\partial\bm n}=0, \quad (\bm x, t)\in \Gamma_2\times [0,T(L-1)],\quad  L \text{ is odd},\\
    \widetilde u_{2,L}^+(\bm x, t)&=\frac{\partial \widetilde u_{2,L}^+(\bm x, t)}{\partial\bm n}=0, \quad (\bm x, t)\in \Gamma_1\times [0,T(L-1)],\quad  L \text{ is even}.
\end{align*}
Then, Definition \ref{def:ms3-bound-con-mul} yields
\begin{align*}
& \widetilde{f}_{1,L+1}=\widetilde{g}_{1,L+1} =0,\quad (\bm x, t)\in \Gamma_1\times [0,T(L-1)],
\quad L+1\;\mbox{is odd},\\
& \widetilde{f}_{2,L+1}=\widetilde{g}_{2,L+1} =0,\quad (\bm x, t)\in \Gamma_2\times [0,T(L-1)],
\quad L+1\;\mbox{is even},
\end{align*}
and hence, the relation \eqref{eq:casulity-trans-con} holds for $M=L+1$. Next, we check the transmission conditions~\eqref{eq:sub-trans1}-\eqref{eq:sub-trans2}. 
On the interface $\Gamma_1$,
\begin{align*}
\widehat u_{2}^M - \widehat u_{1}^M &= \sum_{n=1}^{\lceil M/2 \rceil} \widetilde{u}^{-}_{2,2n-1} + \sum_{n=1}^{\lfloor M/2 \rfloor} \widetilde{u}^{+}_{2,2n} - \sum_{n=1}^{\lceil M/2 \rceil} \widetilde{u}^{+}_{1,2n-1} \\
&= \sum_{n=1}^{\lceil M/2 \rceil}\widetilde f_{1,2n-1} + \sum_{n=1}^{\lfloor M/2 \rfloor} \widetilde{u}^{+}_{2,2n} \\
&=  f_1+\sum_{n=2}^{\lceil M/2 \rceil}\widetilde f_{1,2n-1} + \sum_{n=1}^{\lfloor M/2 \rfloor} \widetilde{u}^{+}_{2,2n}  \\
&=  f_1-\sum_{n=2}^{\lceil M/2 \rceil} \widetilde{u}^{+}_{2, 2n-2} + \sum_{n=1}^{\lfloor M/2 \rfloor} \widetilde{u}^{+}_{2,2n} \\
&= \begin{cases}
f_1, & M \text{ is odd},\\
f_1+\widetilde{u}^{+}_{2,M} , & M \text{ is even}.
\end{cases}
\end{align*}
On the interface $\Gamma_2$,
\begin{align*}
\widehat u_{3}^M - \widehat u_{2}^M &= \sum_{n=1}^{\lfloor M/2 \rfloor} \widetilde{u}^{-}_{3,2n} - \sum_{n=1}^{\lceil M/2 \rceil} \widetilde{u}^{-}_{2,2n-1} - \sum_{n=1}^{\lfloor M/2 \rfloor} \widetilde{u}^{+}_{2,2n} \\
&= \sum_{n=1}^{\lfloor M/2 \rfloor}\widetilde f_{2,2n} - \sum_{n=1}^{\lceil M/2 \rceil} \widetilde{u}^{-}_{2,2n-1} \\
&= f_2 + \widetilde{u}^{-}_{2,1}+\sum_{n=2}^{\lfloor M/2 \rfloor} \widetilde{u}^{-}_{2, 2n-1} - \sum_{n=1}^{\lceil M/2 \rceil} \widetilde{u}^{-}_{2,2n-1}  \\
&= f_2 + \sum_{n=1}^{\lfloor M/2 \rfloor} \widetilde{u}^{-}_{2,2n-1} - \sum_{n=1}^{\lceil M/2 \rceil} \widetilde{u}^{-}_{2,2n-1} \\
&= \begin{cases}
f_2 - \widetilde{u}^{-}_{2,M}, & M \text{ is odd}, \\
f_2, & M \text{ is even}.
\end{cases}
\end{align*}
By the causal condition \eqref{eq:casulity-trans-con} and Corollary~\eqref{DoI-twolayer}, it concludes that $\widetilde{u}^{+}_{2,M}(\bm x, t)=0$ for $(\bm x, t)\in \Gamma_1\times [0,T(M-1)]$ with even $M$ and $\widetilde{u}^{-}_{2,M}(\bm x, t)=0$ for $(\bm x, t)\in \Gamma_2\times [0,T(M-1)]$ with odd $M$, which further implies that \eqref{eq:sub-trans1} holds true. The relation \eqref{eq:sub-trans2} holds following analogous arguments.
\end{proof}

In fact, the ``multiplicative-type'' multiple scattering process discussed above can be further rewritten as a new ``additive-type'' multiple scattering process which will be more efficient to construct parallelizable multiple scattering algorithm for general case $N\ge 3$. To do this, for any $m\in\mathbb{N}$, $j=1,2,3$, we inductively define the boundary-condition functions $\widetilde{f}_{j,m}, \widetilde{g}_{j,m}$ and associated wave-equation solutions
$\widetilde u_{j,m}^+, \widetilde u_{j+1,m}^-$, $m\in\mathbb N$ as follows.

\begin{definition}\label{def:ms3-bound-con-add}
For $m\in \mathbb{N}$ and $N=3$, let $\widetilde u_{j,m}^+, \widetilde u_{j+1,m}^-, j=1,2,3$ be the solutions to the problem (\ref{eq:two-layer1}) where $\widetilde{f}_{j,m}, \widetilde{g}_{j,m}, j=1,2$ are  inductively defined via the relations, for $m=1$,
    \begin{equation}
       (\widetilde f_{j,1}, \widetilde g_{j,1})= (f_j,g_j),\quad j=1,2,
        \label{eq:inc31}
    \end{equation}
and, for $m\ge 2$,
    \begin{align}
       \left( \widetilde f_{1,m}, \widetilde g_{1,m}\right)&=\left( -\widetilde  u_{2,m-1}^+,-\rho_{2}^{-1}\frac{\partial \widetilde  u_{2,m-1}^+}{\partial\bm n}  \right),\\
            \left(\widetilde f_{2,m},\widetilde g_{2,m}\right)&=\left(\widetilde u_{2,m-1}^-, \rho_{1}^{-1} \frac{\partial \widetilde  u_{2,m-1}^-}{\partial\bm n } \right).
        \label{eq:ms3-bound-con-add}
    \end{align}
\end{definition}

Unlike the ``multiplicative-type'' multiple scattering process, in each  ``additive-type'' multiple scattering step, we can solve the two 2-layered sub-problems (\ref{eq:two-layer1}) in parallel, i.e., following the evaluation steps
\begin{align*}
(\widetilde u_{1,1}^+, \widetilde u_{2,1}^-)\;\&\; (\widetilde u_{2,1}^+, \widetilde u_{3,1}^-) \rightarrow (\widetilde u_{1,2}^+, \widetilde u_{2,2}^-)\;\&\;(\widetilde u_{2,2}^+, \widetilde u_{3,2}^-)  \rightarrow ... \rightarrow (\widetilde u_{1,M}^+, \widetilde u_{2,M}^-)\;\&\; (\widetilde u_{2,M}^+, \widetilde u_{3,M}^-),
\end{align*}
and then sum them up in each layer to get a new $M$-th ($M\in\mathbb{N}$) order sum in terms of the solutions
$\widetilde u_{j,m}^+, \widetilde u_{j+1,m}^-$, $j=1,2,3, m\in\mathbb N$ as
 \begin{equation}
\widetilde u_{1}^M = \sum_{m=1}^{M}\widetilde  u_{1,m}^+,\quad \widetilde u_{2}^M = \sum_{m=1}^{M}\widetilde  u_{2,m}^-+\widetilde  u_{2,m}^+,\quad \widetilde u_{3}^M = \sum_{m=1}^{M}\widetilde  u_{3,m}^-.
\label{eq:threelayer-ms-add}
\end{equation}
The solution resulting from the ``additive-type'' multiple scattering process is also equivalent to the solution to the original problem within a finite time interval depending on $M$ which is summarized in the following theorem.
\begin{theorem}
\label{th:ms3-add}
Let $ M \in \mathbb{N} $ with $ M \geq 2 $. Then  we have $ u_j(\bm x, t) = \widetilde u_j^M(\bm x, t) $ for $ (\bm x, t) \in \Omega_j \times [0, T(M)] $, $j=1,2,3$.
\end{theorem}
\begin{proof}
Analogously to the proof of Theorem \ref{th:ms3-mul}, it is easy to show that $\widetilde f_{j,M} = 0$ for $(\bm x, t) \in \Gamma_j \times [0, T(M-1)], j=1,2$.
It suffices to check the transmission condition satisfied by the multiple scattering series solution \eqref{eq:threelayer-ms-add} on $\Gamma_1$ and $\Gamma_2$ for $t\in[0,T(M)]$. By Definition \ref{def:ms3-bound-con-add}, it holds that
\begin{align*}
\widetilde u_{2}^M- \widetilde u_{1}^M & = \sum_{m=1}^{M}\left(\widetilde  u_{2,m}^-+\widetilde  u_{2,m}^+ \right)- \sum_{m=1}^{M}\widetilde  u_{1,m}^+ \\
& = \sum_{m=1}^{M}\left(\widetilde  u_{2,m}^--\widetilde  u_{1,m}^+\right)+\sum_{m=1}^{M}\widetilde  u_{2,m}^+   \\
& =\sum_{m=1}^{M}(\widetilde f_{1,m}-\widetilde f_{1,m+1})   \\
& =\widetilde f_{1,1}-\widetilde f_{1,M+1}
\end{align*}
on the interface $\Gamma_1$ and
\begin{align*}
\widetilde u_{3}^M- \widetilde u_{2}^M & = \sum_{m=1}^{M}\widetilde  u_{3,m}^- - \sum_{m=1}^{M}\left(\widetilde  u_{2,m}^-+\widetilde  u_{2,m}^+\right) \\
& = \sum_{m=1}^{M}\left(\widetilde  u_{3,m}^--\widetilde  u_{2,m}^+\right)-\sum_{m=1}^{M}\widetilde  u_{2,m}^-   \\
& =\sum_{m=1}^{M}\widetilde f_{2,m}-\widetilde f_{2,m+1}                                          \\
& =\widetilde f_{2,1}-\widetilde f_{2,M+1}
\end{align*}
on the interface $\Gamma_2$. Since $\widetilde f_{j,M+1}=0$ for  $(\bm x,t) \in \Gamma_j \times [0, T(M)],j=1,2$, we get
\[
\widetilde u_{j+1}^M -\widetilde  u_{j}^M = f_{j}, \quad  (\bm x, t) \in \Gamma_j \times [0, T(M)], \; j=1,2.
\]
Following a similar discussion for the normal derivatives, we conclude that
\[
\rho_{j+1}^{-1} \frac{\partial \widetilde u_{j+1}^M}{\partial\bm n} - \rho_j^{-1} \frac{\partial \widetilde u_{j}^M}{\partial\bm n} = g_{j}, \quad  (\bm x, t) \in \Gamma_j \times [0, T(M)], \; j=1,2.
\]
The proof is complete.
\end{proof}

Now we successfully transform the original 3-layered medium wave-equation problem into a multiple scattering evaluation of two 2-layered medium wave-equation sub-problems which does not encounter the aforementioned challenge in developing frequency-robust integral equation solver. This strategy can also be extended to the general $N$-layered medium wave-equation problem with $N\ge 3$ which will be discussed in the next section.

\subsection{Multiple scattering: $N\ge 3$}
\label{sec:3.2}

For any $N\ge 3$, following the ``additive-type'' multiple scattering process developed in Section~\ref{sec:3.1} for $N=3$, we can also inductively define the boundary-condition functions $\widetilde{f}_{j,m}, \widetilde{g}_{j,m}$ for $j=1,\cdots,N-1$ and associated wave-equation solutions
$\widetilde u_{j,m}^+, \widetilde u_{j+1,m}^-$, $m\in\mathbb N$ for $j=1,\cdots,N$ as follows.

\begin{definition}\label{def:ms-bound-con}
For $m\in \mathbb{N}$ and $N\ge 3$, let $\widetilde u_{j,m}^+, \widetilde u_{j+1,m}^-, j=1,\cdots,N$ be the solutions to the problem (\ref{eq:two-layer1}) where $\widetilde{f}_{j,m}, \widetilde{g}_{j,m}, j=1,\cdots,N-1$ are inductively defined via the relations, for $m=1$,
\begin{equation}
(\widetilde f_{j,1}, \widetilde g_{j,1})=(f_j, g_j),\quad j=1,\dots,N-1,
\label{eq:inc1}
\end{equation}
and, for $m\ge 2$,
    \begin{align}
        \begin{cases}
            \widetilde f_{j,m}=(1-\delta_{1,j})\widetilde  u_{j,m-1}^--(1-\delta_{N,j+1})\widetilde  u_{j+1,m-1}^+, \\
            \widetilde g_{j,m}=(1-\delta_{1,j})\rho_{j}^{-1} \frac{\partial\widetilde  u_{j,m-1}^-}{\partial \bm n}-(1-\delta_{N,j+1})\rho_{j+1}^{-1} \frac{\partial \widetilde u_{j+1,m-1}^+}{\partial \bm n}.
        \end{cases}
        \label{eq:inc2}
    \end{align}
\end{definition}

Following the evaluation steps
\begin{align*}
\bigcup_{j=1}^{N-1}(\widetilde u_{j,1}^+, \widetilde u_{j+1,1}^-) \rightarrow \bigcup_{j=1}^{N-1}(\widetilde u_{j,2}^+, \widetilde u_{j+1,2}^-)  \rightarrow ... \rightarrow \bigcup_{j=1}^{N-1}(\widetilde u_{j,M}^+, \widetilde u_{j+1,M}^-),
\end{align*}
we introduce an $M$-th ($M\in\mathbb{N}$) order sum in terms of the solutions $\widetilde u_{j,m}^+, \widetilde u_{j+1,m}^-$, $j=1,\cdots,N, m\in\mathbb N$ as
\begin{align}
\widetilde u_{j}^M=
    \sum_{m=1}^{M}\left((1-\delta_{1,j})\widetilde  u_{j,m}^-+(1-\delta_{N,j})\widetilde  u_{j,m}^+\right),\quad & j=1,\dots,N.
    \label{eq:Nlayer-ms-add}
\end{align}
Then we have the following equivalence theorem.

\begin{theorem}\label{thm:equivalence}
Let $ M \in \mathbb{N} $ with $ M \geq 2 $, and define $T(M) = t_0+ M\min_{1 \leq j \leq N-1}c_j^{-1}d_{j,j+1}$. Then the solution to \eqref{eq-transmisson-N}--\eqref{eq-transmisson-condition-N} can be equivalently represented by the multiple scattering series as $u_j(\bm x, t) = \widetilde u_{j}^M(\bm x, t) $ for $ (\bm x, t) \in \Omega_j \times [0, T(M)] $, $j=1,\cdots,N$.
\end{theorem}

\begin{proof}
Analogously to the proof of Theorem \ref{th:ms3-mul}, it is easy to show that $\widetilde f_{j,M} = 0$ for $(\bm x, t) \in \Gamma_j \times [0, T(M-1)]$  by recursion and the same result holds for the boundary data $\widetilde g_{j,M}$. Then it suffices to check the transmission conditions satisfied by the  multiple scattering series solution \eqref{eq:Nlayer-ms-add}. For $ 1 \le j \le N-1 $, Definition \ref{def:ms-bound-con} yields
\begin{align*}
\widetilde u_{j+1}^M-\widetilde u_{j}^M & = \sum_{m=1}^{M}\left((1-\delta_{1,j+1})\widetilde  u_{j+1,m}^-+(1-\delta_{N,j+1})\widetilde  u_{j+1,m}^+\right) - \sum_{m=1}^{M}\left((1-\delta_{1,j})\widetilde  u_{j,m}^-+(1-\delta_{N,j})\widetilde  u_{j,m}^+\right) \\ 
                          & =\sum_{m=1}^{M}\left(\widetilde  u_{j+1,m}^--\widetilde  u_{j,m}^+\right)+\sum_{m=1}^{M}(1-\delta_{N,j+1})\widetilde  u_{j+1,m}^+-\sum_{m=1}^{M}(1-\delta_{1,j})\widetilde  u_{j,m}^-                \\
                          & =\sum_{m=1}^{M}\widetilde f_{j,m}-\widetilde f_{j,m+1}\\
                          & =\widetilde f_{j,1}-\widetilde f_{j,M+1}.
    \end{align*}
Thus, $\widetilde u_{j+1}^M - \widetilde u_{j}^M = f_{j}$ for $(\bm x, t) \in \Gamma_j \times [0, T(M)], j=1,\cdots,N-1$. Analogously, it can be verified that the relation $\rho_{j+1}^{-1} \frac{\partial \widetilde u_{j+1}^M}{\partial\bm n} - \rho_j^{-1} \frac{\partial \widetilde u_{j}^M}{\partial\bm n} = g_{j}$ also holds for $(\bm x, t) \in \Gamma_j \times [0, T(M)], j=1,\cdots,N-1$. These transmission conditions are consistent with \eqref{eq-transmisson-condition-N}, and then the proof is complete.
\end{proof}

\section{FTH-MS integral equation solver}
\label{sec:4}


Based on the multiple scattering re-modeling (Section~\ref{sec:3}) of the original multi-layered medium problem as a union of multiple two-layered medium problems and analogous to the FTH-MS solver for the wave equation problem with multiple bounded obstacle described in Section~\ref{sec:2.2}, the FTH solver developed in \cite{anderson2020high} can now be extended to the numerical evaluation of each two-layered sub-problems which can further induce a novel FTH-MS integral equation solver for the wave equation problem in a general $N$-layered medium. 

\subsection{FTH-MS algorithm}
\label{sec:4.1}

For the two-layered medium problem (\ref{eq:two-layer1}), let $\widetilde U_{j,m}^+(\bm x,\omega)=\mathbb{F}(\widetilde u_{j,m}^+)(\bm x,\omega)$, $\widetilde U_{j+1,m}^-(\bm x,\omega)=\mathbb{F}(\widetilde u_{j+1,m}^-)(\bm x,\omega)$, $\widetilde F_{j,m}(\bm x,\omega)=\mathbb{F}(\widetilde f_{j,m})(\bm x,\omega)$, $\widetilde G_{j,m}(\bm x,\omega)=\mathbb{F}(\widetilde g_{j,m})(\bm x,\omega)$, $F_j(\bm x,\omega)=\mathbb{F}(f_j)(\bm x,\omega)$ and $G_j(\bm x,\omega)=\mathbb{F}(g_j)(\bm x,\omega)$ for $m\in\mathbb{N}$ and $j=1,\cdots,N-1$. Then denoting $\kappa_j=\omega/c_j$, each problem (\ref{eq:two-layer1}) can be transformed into a frequency-domain two-layered medium scattering problem as:
\begin{equation}
    \begin{cases}
        \Delta \widetilde U_{j,m}^+(\bm x,\omega)+\kappa_j^2 \widetilde U_{j,m}^+(\bm x,\omega)=0,                                                                                        & \bm x\in \widetilde{\Omega}_j,     \\
        \Delta \widetilde U_{j+1,m}^-(\bm x,\omega)+\kappa_{j+1}^2 \widetilde U_{j+1,m}^-(\bm x,\omega)=0,                                                                                    & \bm x\in \widetilde{\Omega}_{j+1}, \\
        \widetilde U_{j+1,m}^-(\bm x,\omega)-\widetilde U_{j,m}^+(\bm x,\omega)=\widetilde F_{j,m}(\bm x,\omega),                                                                                  & \bm x\in \Gamma_j,             \\
        \rho_{j+1}^{-1}\frac{\partial \widetilde U_{j+1,m}^-(\bm x,\omega)}{\partial \bm n}-\rho_{j}^{-1}\frac{\partial \widetilde U_{j,m}^+(\bm x,\omega)}{\partial \bm n}=\widetilde G_{j,m}(\bm x,\omega), & \bm x\in \Gamma_j,
    \end{cases}
    \label{eq:freq-2-layer}
\end{equation}
which can be shown to be well-posed~\cite{Zhang_1994} imposing appropriate Sommerfeld radiation condition.
Analogous to Definition~\ref{def:ms3-bound-con-mul}, the boundary data $\widetilde F_{j,m}, \widetilde G_{j,m}$ in frequency-domain should be inductively defined via the relations, for $m=1$,
\begin{equation}
(\widetilde F_{j,1}, \widetilde G_{j,1})=(F_j, G_j),\quad j=1,\dots,N-1,
\label{eq:inc1-FD}
\end{equation}
and, for $m\ge 2$,
    \begin{align}
        \begin{cases}
            \widetilde F_{j,m}=(1-\delta_{1,j})\widetilde  U_{j,m-1}^--(1-\delta_{N,j+1})\widetilde  U_{j+1,m-1}^+, \\
            \widetilde G_{j,m}=(1-\delta_{1,j})\rho_{j}^{-1} \frac{\partial\widetilde  U_{j,m-1}^-}{\partial \bm n}-(1-\delta_{N,j+1})\rho_{j+1}^{-1} \frac{\partial \widetilde U_{j+1,m-1}^+}{\partial \bm n}.
        \end{cases}
        \label{eq:inc2-FD}
    \end{align}

Then utilizing appropriate Fourier transform algorithm (see Section~\ref{sec:4.2}) and frequency-domain integral equation solver (see Section~\ref{sec:4.3}), the original wave equation problem can be solved by noting from Theorem~\ref{thm:equivalence} that $u_j(\bm x, t) = \widetilde u_{j}^M(\bm x, t) $ for $ (\bm x, t) \in \Omega_j \times [0, T(M)] $, $j=1,\cdots,N$ through the so-called FTH-MS integral equation solver summarized in Algorithm~\ref{alg:msN}. It should be pointed out that the numerical implementation of Algorithm~\ref{alg:msN} require several contents of numerical techniques as follows:
\begin{itemize}
\item A truncation of the frequency interval depending on the frequency spectrum of the incident signal;
\item Appropriate direct and inverse Fourier transform algorithm;
\item Appropriate boundary integral equation solver for the problem (\ref{eq:freq-2-layer}).
\end{itemize}
These will be discussed in the following Sections~\ref{sec:4.2}-\ref{sec:4.3}.

\begin{algorithm}
    \caption{FTH-MS solver for the problem (\ref{eq-transmisson-N}) in an $N$-layered medium}
    \label{alg:msN}
    \begin{algorithmic}[1]
        \STATE Initialize $\widetilde u_{j}^M(\bm x,t)=0, (\bm x,t)\in \Omega_j\times [0,T]$ for $j=1,\cdots,N$.
        \STATE Do $m=1,\cdots,M$
        \STATE \ \ \ \ Do $j=1,\cdots,N-1$
        \STATE \ \ \ \ \ \ \ \ If $m=1$
        \STATE \ \ \ \ \ \ \ \ \ \ \ \ Evaluate the boundary data $\widetilde F_{j,m}(\bm x,\omega), \widetilde G_{j,m}(\bm x,\omega), (\bm x,\omega)\in\Gamma_j\times\R$ via relation (\ref{eq:inc1-FD}).
        \STATE \ \ \ \ \ \ \ \ Else
        \STATE \ \ \ \ \ \ \ \ \ \ \ \ Evaluate the boundary data $\widetilde F_{j,m}(\bm x,\omega), \widetilde G_{j,m}(\bm x,\omega), (\bm x,\omega)\in\Gamma_j\times\R$ via relations (\ref{eq:inc2-FD}).
        \STATE \ \ \ \ \ \ \ \ End If
        \STATE \ \ \ \ \ \ \ \ Solve the problem (\ref{eq:freq-2-layer}) using appropriate boundary integral equation solver.
        \STATE \ \ \ \ \ \ \ \ Compute $\widetilde u_j^M(\bm x,t)\mathrel{+}= \mathbb{F}^{-1}(\widetilde U_{j,m}^+)(\bm x,t), (\bm x,t)\in\Omega_j\times [0,T]$.
        \STATE \ \ \ \ \ \ \ \ Compute $\widetilde u_{j+1}^M(\bm x,t)\mathrel{+}= \mathbb{F}^{-1}(\widetilde U_{j+1,m}^-)(\bm x,t), (\bm x,t)\in\Omega_{j+1}\times [0,T]$.
        \STATE \ \ \ \ End Do
        \STATE End Do
    \end{algorithmic}
\end{algorithm}

\subsection{Specified incident pulse and Fourier transform algorithm}
\label{sec:4.2}

In this work, the incident wave considered in the numerical evaluation is specified, for simplicity, as a Gaussian pulse for which the point source $u^{\text{inc}}(\bm x,t)=u_{\text{point}}^{\text{inc}}(\bm x,t;\bm z_0)$ located at $\bm z_0\in\Omega_{j_0}$ is an inverse Fourier transform of a frequency-domain function
\begin{align*}
U_{\text{point}}^{\text{inc}}(\bm x,\omega,\bm z_0)&=\frac{5i}{2} e^{-\frac{(\omega-\omega_0)^2}{2}} e^{i\omega t_0} \cdot H_0^{(1)}(\kappa_{j_0} |\bm x-\bm z_0|),
\end{align*}
i.e., $u_{\text{point}}^{\text{inc}}(\bm x,t;\bm z_0)=\mathbb F^{-1}(U_{\text{point}}^{\text{inc}})(\bm x,t;\bm z_0)$, and the plane wave $u^{\text{inc}}(\bm x,t)=u_{\text{plane}}^{\text{inc}}(\bm x,t;d^{\text{inc}})$ is, analogously, an inverse Fourier transform of another frequency-domain function
\begin{align*}
U_{\text{plane}}^{\text{inc}}(\bm x, \omega) = \sqrt{2\pi} e^{-\frac{(\omega-\omega_0)^2}{2}} e^{i\omega t_0} e^{i\kappa_{1} \bm x \cdot d^{\text{inc}}}.
\end{align*}
i.e., $u_{\text{plane}}^{\text{inc}}(\bm x,t;\bm z_0)=\mathbb F^{-1}(U_{\text{plane}}^{\text{inc}})(\bm x,t;\bm z_0)$. Here, $\omega_0$ denotes the center frequency, $t_0$ is the time delay, $\bm z_0$ is the location of point source and $d^{\text{inc}}=(\cos\theta^{\text{inc}},\sin\theta^{\text{inc}})^\top$ denotes the incident direction with $\theta^{\text{inc}}\in(-\pi/2,\pi/2)$ being the incident angle. Here we ignore the discussion of critical angle under which total reflection happens. 

Imposing the specified incident fields and noting the exponentially decaying property of the Gaussian pulse as $\omega\rightarrow\infty$, it suffices to evaluate the frequency-domain solutions $\widetilde U_{j,m}^+(\bm x,\omega)$ and $\widetilde U_{j+1,m}^-(\bm x,\omega)$ for $\omega\in [\omega_0-W,\omega_0+W]$ with $W>0$ being selected big enough such that the absolute values of $\widetilde U_{j,m}^+(\bm x,\omega)$ and $\widetilde U_{j+1,m}^-(\bm x,\omega)$ are below machine precision when $\omega\notin [\omega_0-W,\omega_0+W]$. Then the inverse Fourier transform of $\widetilde U_{j,m}^+$, as well as the inverse Fourier transform of $\widetilde U_{j+1,m}^-$, can be approximated by
\begin{equation}
\label{eq:truncated-Fourier}
\mathbb{F}^{-1}(\widetilde U_{j,m}^+)(\bm x,t)\approx\frac{1}{2\pi}\int_{\omega_0-W}^{\omega_0+W}\widetilde U_{j,m}^+(\bm x,\omega)e^{-i\omega t}d\omega,
\end{equation}
with high accuracy. To avoid the discussion of low-frequency and long-time simulation issues, see Remarks~\ref{rem:low-freq} and \ref{rem:long-time-stability}, which will be left for future works, we assume that $\omega_0-W>0$. Then the numerical evaluation of \eqref{eq:truncated-Fourier} can be achieved through FFT.

\begin{remark}[Discussions on low-frequency case]
\label{rem:low-freq}
For general cases of incident fields containing information on a wide frequency interval $\omega\in [-\widetilde W,\widetilde W], \widetilde W>0$, the numerical approximation of the frequency-domain solutions and inverse Fourier transform will bring significant challenges. The later case comes from the fact that the two-dimensional frequency-domain solutions of wave scattering problems generally contained integrable $\mathcal{O}(\log\omega)$ singularities \cite{maccamy1965low,werner1986low}. The calculation of the inverse Fourier transform on $\omega\in(-\omega_c,\omega_c)$ with $\omega_c>0$ being some small value of frequency can be resolved utilizing the midified ``Filon-Clenshaw-Curtis'' high-order quadrature rule developed in \cite{anderson2020high}. However, the numerical evaluation of the frequency-domain solutions for layered-medium scattering problems still remains challenging. In particular, the PML-BIE method discussed in the next Section~\ref{sec:4.3} does not provide a uniformly high-accuracy solver due to the fact that the convergence of the PML truncation, in which the thickness of PML layer requires to be inversely proportional to the wave number, will be destroyed. One way to overcome this difficulty is to deform the integral above the real axis~\cite{chew1999waves}, however, will bring an exponential growth factor in the inverse Fourier transform depending on the largest distance of the new integral path to zero. This will be briefly discussed in numerical Example~3. Using the classical layered Green function method \cite{Cai2002} to solve the low-frequency problems might be another alternative choice, however, needs efficient algorithms for the computing of expensive Sommerfeld integrals~\cite{cai00}. 
\end{remark}

\begin{remark}[Discussions on long-time simulation]
\label{rem:long-time-stability}

Following idea of applying ``time windowing and recentering'' approach to achieve long-time simulation proposed in~\cite{anderson2020high}, the multiple scattering Algorithm \ref{alg:msN} can be further extended to a modified version for the purpose of long-time simulation for the considered multi-layered medium wave-equation problem. We use the notations in~\cite{pan2025multi}. For a given large final time $T$, let $\mathcal{P}=\{\sqcap_q(t)\ |\ q\in\mathcal{Q}\},
\mathcal{Q}=\{1,\cdots,Q\}$ be a smooth partition of unity such that $\sum_{q\in\mathcal{Q}}\sqcap_q(t)\ = 1$ for $t\in[0,T]$ where for a certain sequence $s_q$ ($q\in\mathcal{Q}$), each $\sqcap_q(t)$ is a non-negative, smooth, locally supported windowing function of $t$ satisfying $\mathrm{supp}\,\sqcap_q\subset[s_q-H,s_q+H]$. Utilizing the partition of unity $\mathcal{P}$ leads to an expression, for any smooth long-time
signal $g(t)$, $t\in[0,T]$, as
\begin{equation}\label{eq:windowing}
  g(t)=\sum_{q\in\mathcal{Q}} g_q(t), \quad g_q(t)=g(t)\sqcap_q(t),
\end{equation}
where $\mathrm{supp}\, g_q\subset[s_q-H,s_q+H]$. Then the
corresponding Fourier transform $G=\mathbb F(g)$ can be expressed by
\begin{equation}
G(\omega)=\sum_{q\in\mathcal{Q}} G_q(\omega),\quad
G_q(\omega)=\int_{0}^{T} g_q(t)e^{i\omega t}dt= e^{i\omega s_q}G_{q,slow}(\omega),
\end{equation}
where $G_{q,slow}(\omega)= \mathbb{F}_{q,slow}(g)(\omega)$ with
\begin{equation}
\label{Fourierk}
\mathbb{F}_{q,slow}(g)(\omega):=\int_{-H}^H g(t+s_q)\sqcap_q(t)e^{i\omega t}dt,
\end{equation}
denoting the $s_q$-centered slow Fourier-transform operator. Note that each $G_{q,slow}$ now is a slowly-oscillatory function of $\omega$. For the two-layered medium problem (\ref{eq:two-layer1}) and the frequency-domain problem (\ref{eq:freq-2-layer}), we denote $\widetilde F_{j,m,q,slow}(\bm x,\omega)= \mathbb{F}_{q,slow}(\widetilde f_{j,m})(\bm x,\omega)$ and $\widetilde G_{j,m,q,slow}(\bm x,\omega)= \mathbb{F}_{q,slow}(\widetilde g_{j,m})(\bm x,\omega)$. Now let $\widetilde{U}_{j,m,q,slow}^+, \widetilde{U}_{j+1,m,q,slow}^-$ be the solutions to the frequency-domain problem (\ref{eq:freq-2-layer}) with boundary data $\widetilde F_{j,m,q,slow},\widetilde G_{j,m,q,slow}$. Then the solutions $\widetilde{u}_{j,m}^+, \widetilde{u}_{j+1,m}^-$ to the time-domain two-layered medium problem (\ref{eq:two-layer1}) can be expressed as 
\begin{equation}
\label{eq:long-sol1}
\widetilde{u}_{j,m}^+(\bm x,t)=\mathbb{F}^{-1}\left( \sum_{q\in\mathcal{Q}} e^{i\omega s_q}\widetilde{U}_{j,m,q,slow}^+\right)(\bm x,t) =\sum_{q\in\mathcal{Q}} \mathbb{F}^{-1}(\widetilde{U}_{j,m,q,slow}^+)(\bm x,t-s_q),
\end{equation}
and, analogously,
\begin{equation}
\label{eq:long-sol2}
\widetilde{u}_{j+1,m}^-(\bm x,t)=\sum_{q\in\mathcal{Q}} \mathbb{F}^{-1}(\widetilde{U}_{j+1,m,q,slow}^-)(\bm x,t-s_q).
\end{equation}
It should be noticed that even for the specified Gauss pulse whose vast majority of information is concentrated on interval $[\omega_0-W,\omega_0+W], \omega_0-W>0$, the frequency band of the slowly-oscillatory functions $\widetilde F_{j,m,q,slow}, \widetilde G_{j,m,q,slow}$ will be shifted and enlarged due to the windowing function. This implies that
the calculation of the inverse Fourier transform in (\ref{eq:long-sol1})-(\ref{eq:long-sol2}) requires to be imposed on a wider frequency range $\omega\in [-\widetilde W,\widetilde W], \widetilde W>0$ to achieve high accuracy and thus, an efficient numerical solver for the low-frequency problems will also be necessary. Both the low-frequency case and long-time simulation will be left for future work.
\end{remark}

\subsection{PML-BIE method for frequency-domain problems}
\label{sec:4.3}

Finally, we briefly discuss the boundary integral equation method for solving the problem (\ref{eq:freq-2-layer}) with $\omega>0$ where the main difficulty comes from the infinite interface. Here we utilized the so-called PML-BIE method proposed in~\cite{Lu2018Perfectly} (for highly accurate numerical discretization of the PML-transformed boundary integral operators, we refer to~\cite{Lu2023HighlyPML}). The exponential convergence of the PML truncation for the two-layered medium problem has been proved in~\cite{chen2010}. In fact, the PML acts as an artificial absorbing layer that damps outgoing waves without causing reflections, which is based on a complex coordinate stretching defined as follows:
\begin{equation*}
\widetilde{x}_l=x_l+i\int_{0}^{x_l}\sigma_l(s)ds,\quad l=1,2,
\end{equation*}
where $\sigma_l(s)$ is a positive function. Here, we use the PML transformation in~\cite{Lu2018Perfectly} and the function $\sigma_l(s)$ is given by
\begin{equation*}
    \sigma_l(s)=\begin{cases}
        0,\quad                           & -a_l\leq s\leq a_l,    \\
        \frac{2Sf_1^P(s)}{f_1^P(s)+f_2^P(s)},\quad & a_l\leq s\leq a_l+T_l, \\
        S,\quad                           & s> a_l+T_l,            \\
        \sigma_l(-s),\quad                & s<-a_l,
    \end{cases}
\end{equation*}
where $S>0$ is a constant, $P$ is a positive integer, $f_1(s)=(\tfrac{1}{2}-\tfrac{1}{P})\tilde{s}^3+\tfrac{\tilde{s}}{P}+\frac{1}{2},f_2(s)=1-f_1(s) ,\tilde{s}=\frac{s-(a_l+T_l)}{T_l}$, $T_l$ is the thickness of PML layer, $2a_l$ is the size of Physical domain.


\begin{figure}
    \centering
    \includegraphics[width=0.3\linewidth]{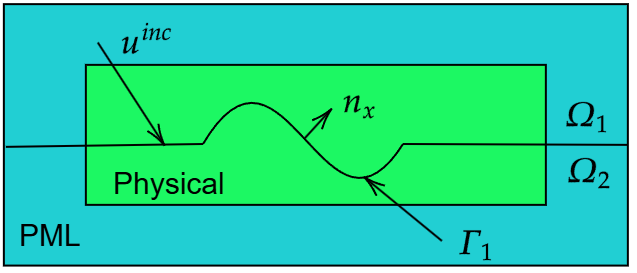}
    \caption{Physical domain(green) surrounded by PML(blue) for two-layered media}
    \label{fig:PML}
\end{figure}

The PML-transformed free-space fundamental solution, denoted by $\widetilde{G}_\kappa(\bm x,\bm y)$, is obtained by applying the complex coordinate stretching to the arguments of the free-space fundamental solution $G_\kappa(\bm x,\bm y)=\tfrac{i}{4}H_0^{(1)}(\kappa|\bm x-\bm y|)$~\cite{Lu2018Perfectly}, i.e.,
\begin{equation*}
    \label{eq:pml-green-function}
    \widetilde{G}_{\kappa}(\bm x, \bm y) := G_{\kappa}(\widetilde{\bm x}, \widetilde{\bm y}) = \frac{i}{4} H_0^{(1)}(\kappa \rho(\widetilde{\bm x}, \widetilde{\bm y})),\quad \bm x, \bm y \in \mathbb{R}^2,
\end{equation*}
where $\rho$ denotes the complex-valued distance function given by
\begin{equation*}
    \label{eq:complex-distance}
    \rho(\widetilde{\bm x}, \widetilde{\bm y}) = \sqrt{(\widetilde{x}_1 - \widetilde{y}_1)^2 + (\widetilde{x}_2 - \widetilde{y}_2)^2},
\end{equation*}
in which the branch of the complex square root $\sqrt{z}$ is chosen such that $\operatorname{Re}(\sqrt{z}) \ge 0$ for any $z \in \mathbb{C}$.

As shown in Figure~\ref{fig:PML}, utilizing the PML truncation and imposing the vanishing conditions of both the Dirichlet and Neumann data on the exterior boundary of PML region, the PML-BIE method seeks approximate solutions in the truncated regions $\widetilde\Omega_j^\mathrm{pml}:=\{\bm x\in\widetilde\Omega_{j}: |x_l|<a_l+T_l, l=1,2\}$ and
$\widetilde\Omega_{j+1}^\mathrm{pml}:=\{\bm x\in\widetilde\Omega_{j+1}: |x_l|<a_l+T_l, l=1,2\}$ for the solutions $\widetilde{\widetilde{U}}_{j,m}^+(\bm x)= \widetilde{U}_{j,m}^+(\widetilde{\bm x})$ and $\widetilde{\widetilde{U}}_{j+1,m}^-(\bm x)= \widetilde{U}_{j+1,m}^-(\widetilde{\bm x})$ through the representations (see for example \cite{Lu2018Perfectly,Lu2023HighlyPML})
\begin{equation}
    \begin{cases}
\widetilde{\widetilde{U}}_{j,m}^+(\bm x,\omega)    =\left(D_{j,\kappa_j} \widetilde{\widetilde{U}}_{j,m}^+\right)(\bm x,\omega)-\left(S_{j,\kappa_j}\frac{\partial \widetilde{\widetilde{U}}_{j,m}^+}{\partial\bm n}\right)(\bm x,\omega),\quad        & \bm x\in\widetilde\Omega_j^\mathrm{pml},    \\
\widetilde{\widetilde{U}}_{j+1,m}^-(\bm x,\omega)    =-\left(D_{j,\kappa_{j+1}} \widetilde{\widetilde{U}}_{j+1,m}^-\right)(\bm x,\omega)+\left(S_{j,\kappa_{j+1}}\frac{\partial \widetilde{\widetilde{U}}_{j+1,m}^-}{\partial\bm n}\right)(\bm x,\omega),\quad & \bm x\in\widetilde\Omega_{j+1}^\mathrm{pml}.
\end{cases}
\label{eq:green-representation}
\end{equation}
where $S_{j,\kappa},D_{j,\kappa}$ are the single-layer and double-layer potentials operators defined on the PML-truncated interface $\Gamma_{j}^\mathrm{pml}:=\{\bm x\in\Gamma_j: |x_1|\le a_1+T_1\}$, given by
by
\[
(S_{j,\kappa}q)(\bm x) = \int_{\Gamma_{j}^\mathrm{pml}} \widetilde{G}_{\kappa}(\bm x, \bm y) q(\bm y) \, ds_{\bm y},
\]
\[
(D_{j,\kappa}q)(\bm x) = \int_{\Gamma_{j}^\mathrm{pml}} \frac{\partial}{\widetilde{\partial} \bm n_{\bm y}} \widetilde{G}_{\kappa}(\bm x, \bm y) q(\bm y) \, ds_{\bm y}.
\]
Here, $\frac{\partial}{\widetilde{\partial} \bm n} u=\mathbb{A} \bm n\cdot \nabla u
$ with $\mathbb{A}=\mathrm{diag}\{a_1^{-1}(x_1)a_2(x_2),a_1(x_1)a_2^{-1}(x_2)\}$ and $a_i(x_i)=\widetilde{x_i}'(x_i)$. Then taking the limit $\bm x\rightarrow \Gamma_{j}^\mathrm{pml}$ in (\ref{eq:green-representation}) and applying the PML transformed boundary conditions 
\begin{equation*}
\begin{cases}
\widetilde{\widetilde{U}}_{j+1,m}^-(\bm x,\omega)-\widetilde{\widetilde{U}}_{j,m}^+(\bm x,\omega)=\widetilde{\widetilde{F}}_{j,m}(\bm x,\omega)(:=\widetilde F_{j,m}(\widetilde{\bm x},\omega)), & \bm x\in\Gamma_{j}^\mathrm{pml},             \\
\rho_{j+1}^{-1}\frac{\partial \widetilde{\widetilde{U}}_{j+1,m}^-(\bm x,\omega)}{\widetilde\partial \bm n}-\rho_{j}^{-1}\frac{\partial \widetilde{\widetilde{U}}_{j,m}^+(\bm x,\omega)}{\widetilde\partial \bm n}=\widetilde{\widetilde{G}}_{j,m}(\bm x,\omega)(:=\widetilde G_{j,m}(\widetilde{\bm x},\omega)), & \bm x\in \Gamma_{j}^\mathrm{pml},
\end{cases}
\end{equation*}
leads to the BIE system, analogous to~\cite{Bruno2016WindowedGF}, 
\begin{equation}
(\mathcal{E}+\mathcal{T})\Phi=\Phi^i,
\label{eq:bie}
\end{equation}
in terms of the boundary integral operators
\[
    (\mathcal{S}_{j,\kappa}q)(\bm x) = \int_{\Gamma_{j}^\mathrm{pml}} \widetilde{G}_\kappa(\bm x, \bm y) q(\bm y) \, ds_{\bm y}, \quad (\mathcal{K}_{j,\kappa}q)(\bm x) = \int_{\Gamma_{j}^\mathrm{pml}} \frac{\partial}{\widetilde{\partial} \bm n_{\bm y}} \widetilde{G}_\kappa(\bm x, \bm y) q(\bm y) \, ds_{\bm y},
\]
\[
    (\mathcal{K}_{j,\kappa}'q)(\bm x) = \int_{\Gamma_{j}^\mathrm{pml}} \frac{\partial}{\widetilde{\partial} \bm n_{\bm x}} \widetilde{G}_\kappa(\bm x, \bm y) q(\bm y) \, ds_{\bm y}, \quad (\mathcal{N}_{j,\kappa}q)(\bm x) = \int_{\Gamma_{j}^\mathrm{pml}} \frac{\partial^2}{\widetilde{\partial} \bm n_{\bm x} \widetilde{\partial} \bm n_{\bm y}} \widetilde{G}_\kappa(\bm x, \bm y) q(\bm y) \, ds_{\bm y},
\]
where
\begin{equation*}
    \mathcal{E}=\begin{pmatrix}
        \mathcal{I} & 0           \\
        0           & \mathcal{I}
    \end{pmatrix},\quad \mathcal{T}=\begin{pmatrix}
        \mathcal{K}_{j,\kappa_{j+1}}-\mathcal{K}_{j,\kappa_j}   & \mathcal{S}_{j,\kappa_j}-\mathcal{S}_{j,\kappa_{j+1}}   \\
        \mathcal{N}_{j,\kappa_{j+1}}-\mathcal{N}_{j,\kappa_{j}} & \mathcal{K}_{j,\kappa_j}'-\mathcal{K}_{j,\kappa_{j+1}}'
    \end{pmatrix},
\end{equation*}
and
\begin{equation*}
\Phi=\begin{pmatrix}
        \widetilde{\widetilde{U}}_{j+1,m}^- \\
        \frac{\partial \widetilde{\widetilde{U}}_{j+1,m}^-}{\widetilde{\partial} \bm n}
    \end{pmatrix},\quad \Phi^i=\begin{pmatrix}
        \frac{\widetilde{\widetilde{F}}_{j,m}}{2}+ \mathcal{S}_{j,\kappa_j}\widetilde{\widetilde{G}}_{j,m}-\mathcal{K}_{j,\kappa_j}\widetilde{\widetilde{F}}_{j,m} \\
        \frac{\widetilde{\widetilde{G}}_{j,m}}{2}+ \mathcal{K}_{j,\kappa_j}'\widetilde{\widetilde{G}}_{j,m}-\mathcal{N}_{j,\kappa_j}\widetilde{\widetilde{F}}_{j,m}
    \end{pmatrix}.
\end{equation*}
with $\mathcal{I}$ denoting the identity operator. 

The numerical solution of the BIE system \eqref{eq:bie} requires careful treatment of its singular integral kernels. The integral operators in the BIE systems can be written in the following general form:
\begin{equation}
\label{eq:general-integral-op-final}
(\mathcal{L}q)(\bm x) = \int_{\Gamma_{j}^\mathrm{pml}} L(\bm x,\bm y) q(\bm y) \, ds(\bm y),\quad \bm x\in \Gamma_{j}^\mathrm{pml},
\end{equation}
where $L(\bm x,\bm y)$ is the kernel and $q$ is the density function. To handle the singularities, the Chebyshev-based rectangular-polar solver proposed in \cite{BRUNO2020109350,Lu2023HighlyPML}, which provides an efficient algorithm to achieve high accuracy, is utilized for the numerical evaluation of (\ref{eq:general-integral-op-final}). In addition, both the multiple scattering iterations and the recovery of the time-domain solution require solving the BIE system for a large number of frequencies. To make this computationally feasible, we employ a kernel expansion method, as detailed in \cite{bruno2024multiple}. This technique accelerates the process by decomposing the kernel into a sum of two parts: a singular but frequency-independent part, and a regular (smooth) but frequency-dependent part, thus allowing us to compute the singular part only once and reuse it for all frequencies.

\section{Numerical experiments}
\label{sec:5}

In this section, several numerical examples will be presented to show the efficiency and accuracy of the proposed FTH-MS solver for the the wave equation problem in a multi-layered medium in both two and three dimensions, including the two-dimensional 2/3/5-layered problems, as well as the three-dimensional 3-layered problem. To quantify the accuracy, we measure the maximum error $\varepsilon_\infty$ in the time-domain solution for $t\in[0, T]$ at a set of specified observation points $\{\bm x_j, j=1,\cdots,N_{\mathrm{obs}}\}$ via
\begin{equation*}
\label{eq:error_definition}
\varepsilon_\infty(\bm x_j) := \max_{t \in [0,T]} \varepsilon(\bm x_j,t)\quad\mbox{with}\quad \varepsilon(\bm x_j,t)=|u_{\text{num}}(\bm x_j, t) - u_{\text{ref}}(\bm x_j, t)|.
\end{equation*}
Here, $u_{\text{num}}$ denotes the numerical solution resulting from the proposed FTH-MS method, and $u_{\text{ref}}$ is a reference solution generated from either the exact solution or the numerical solution with a significantly finer discretization in cases where a closed-form exact solution is not available. All numerical results were obtained from a C++ implementation of the described algorithms. The computations were performed on a standard laptop with an AMD Ryzen 7 8845H processor and 32 GB of RAM.

\begin{figure}[htbp]
    \centering
    \subfloat[]{
        \includegraphics[width=0.2\textwidth]{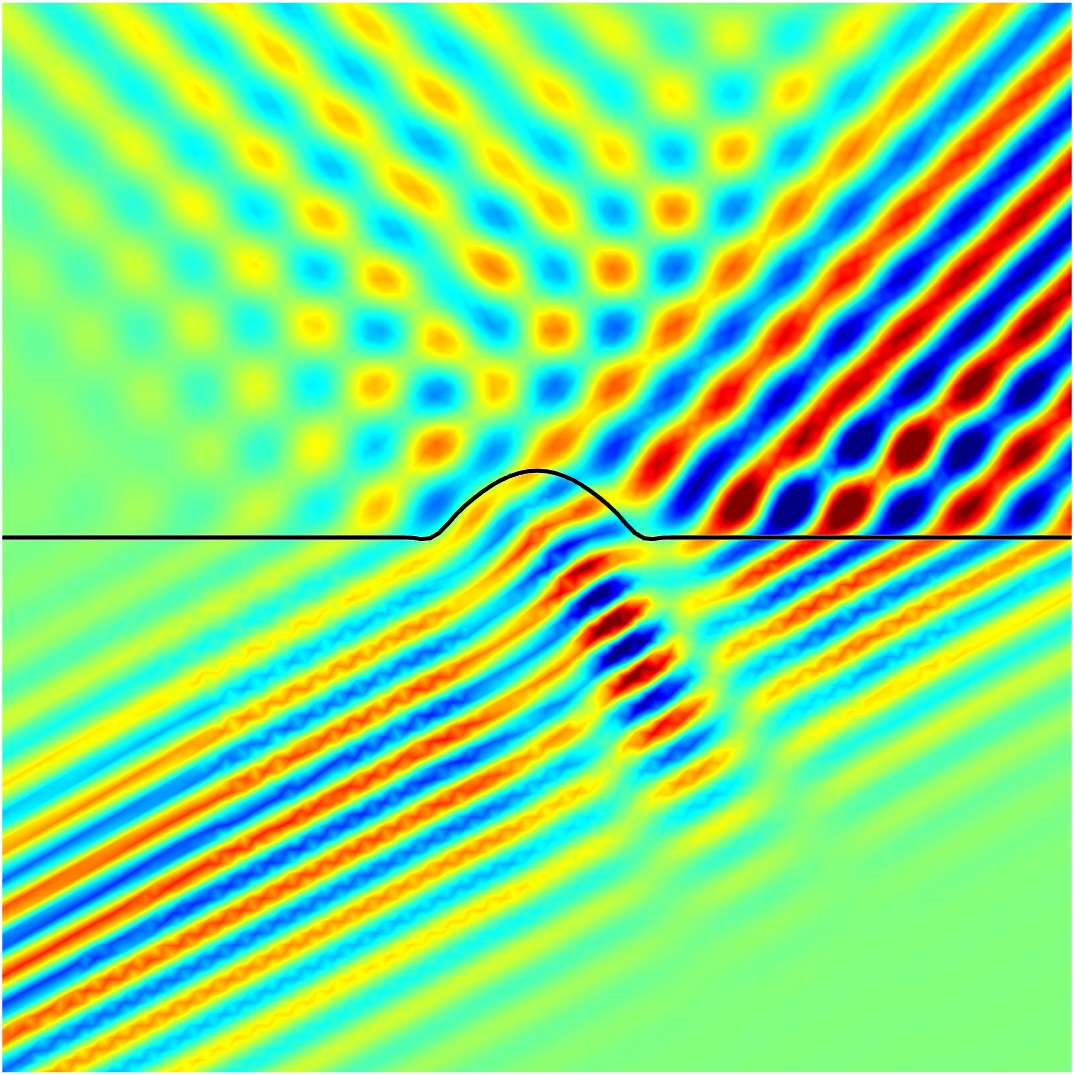}
    }
    \subfloat[]{
        \includegraphics[width=0.25\textwidth]{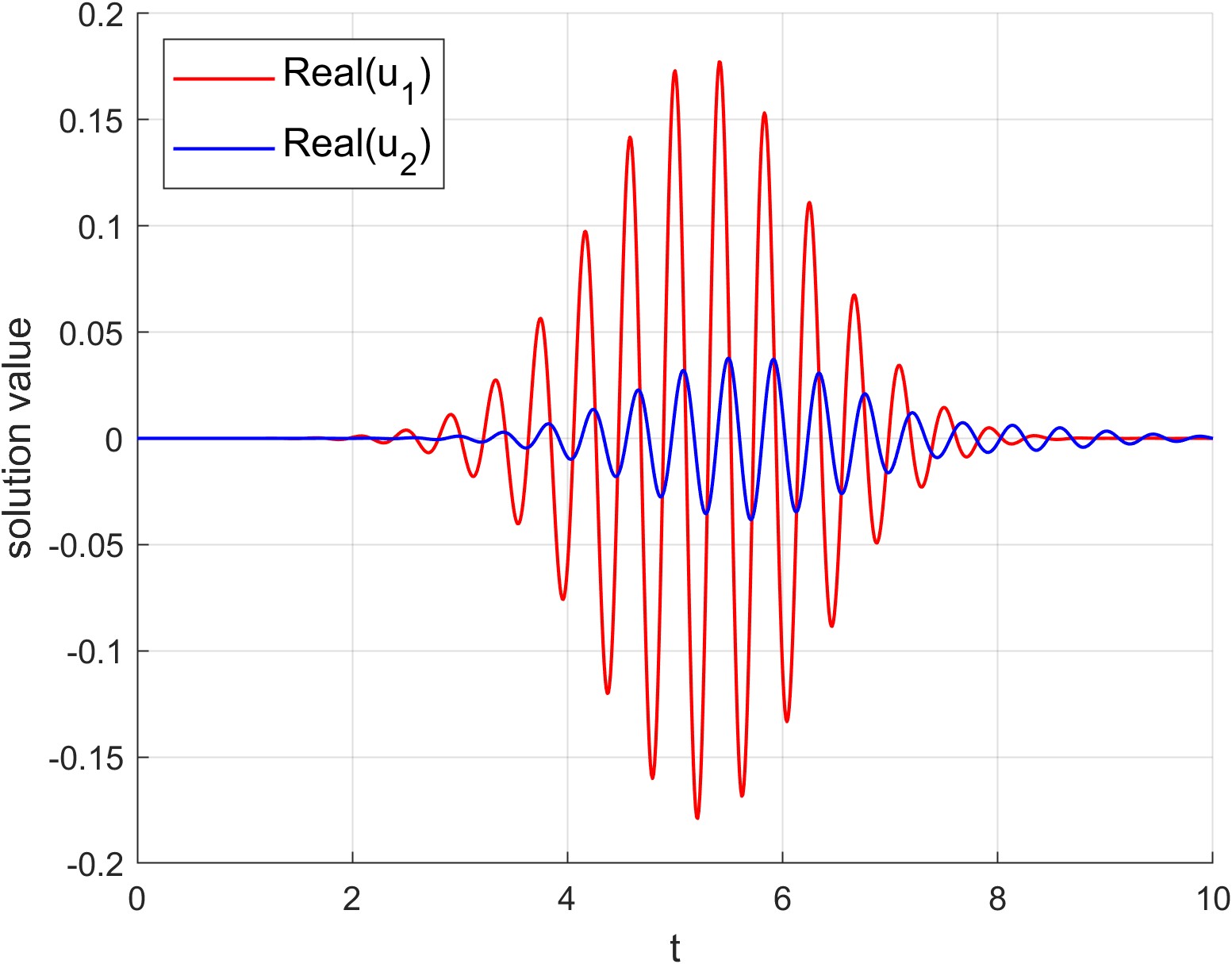}
    }
    \subfloat[]{
        \includegraphics[width=0.25\textwidth]{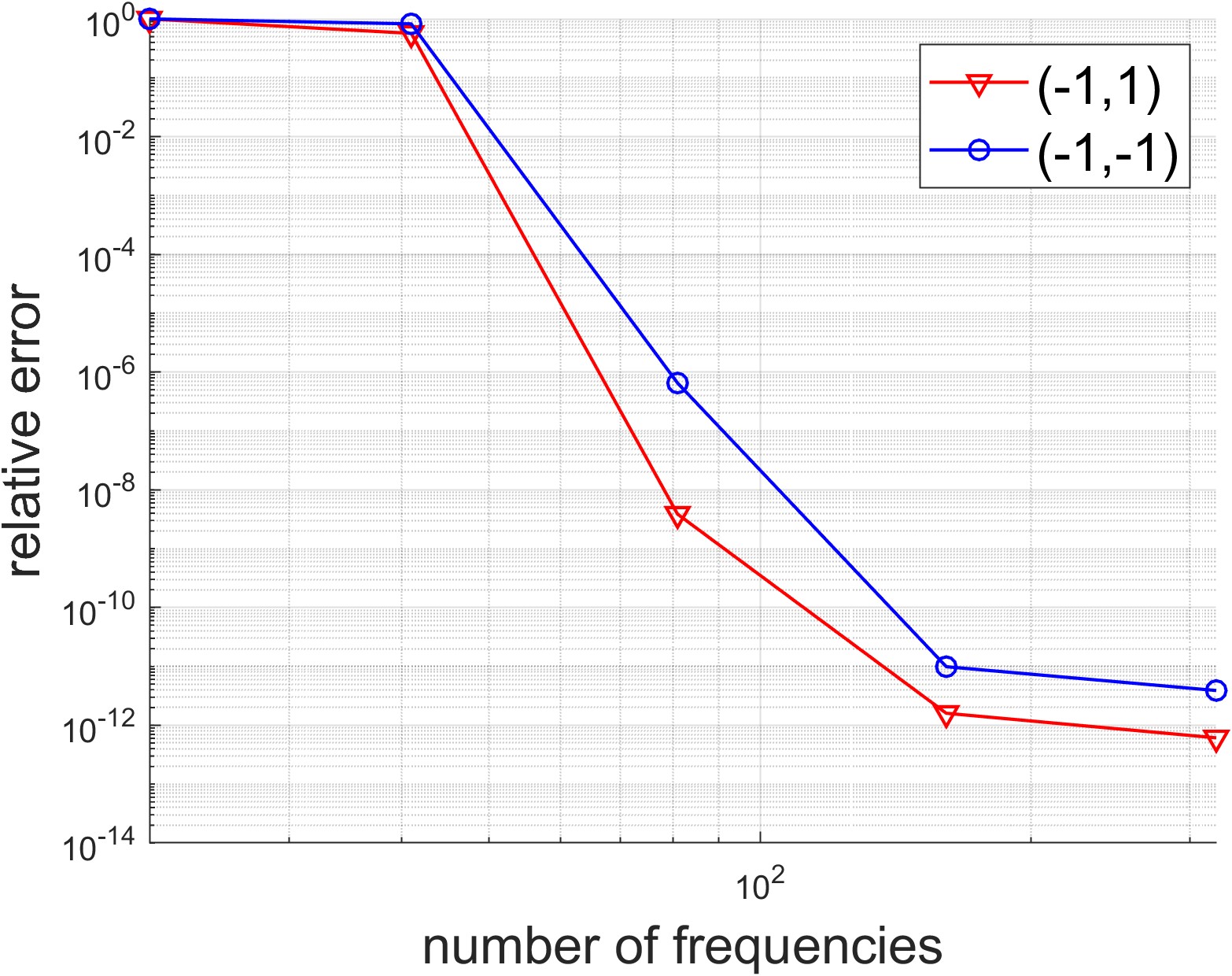}
    }
    \caption{Example 1. Numerical results of the FTH method for the 2-layered medium problem with a plane wave incidence:
    (a) total fields at $t=6$; (b) time trace of the scattered fields at $\bm x=(-1,1)$ and $\bm x=(-1,-1)$; and (c) maximum errors $\varepsilon_\infty$ as a function of the number of frequencies.}
    \label{fig:example1-plane}
\end{figure}

\begin{figure}[htbp]
    \centering
    \subfloat[]{
        \includegraphics[width=0.2\textwidth]{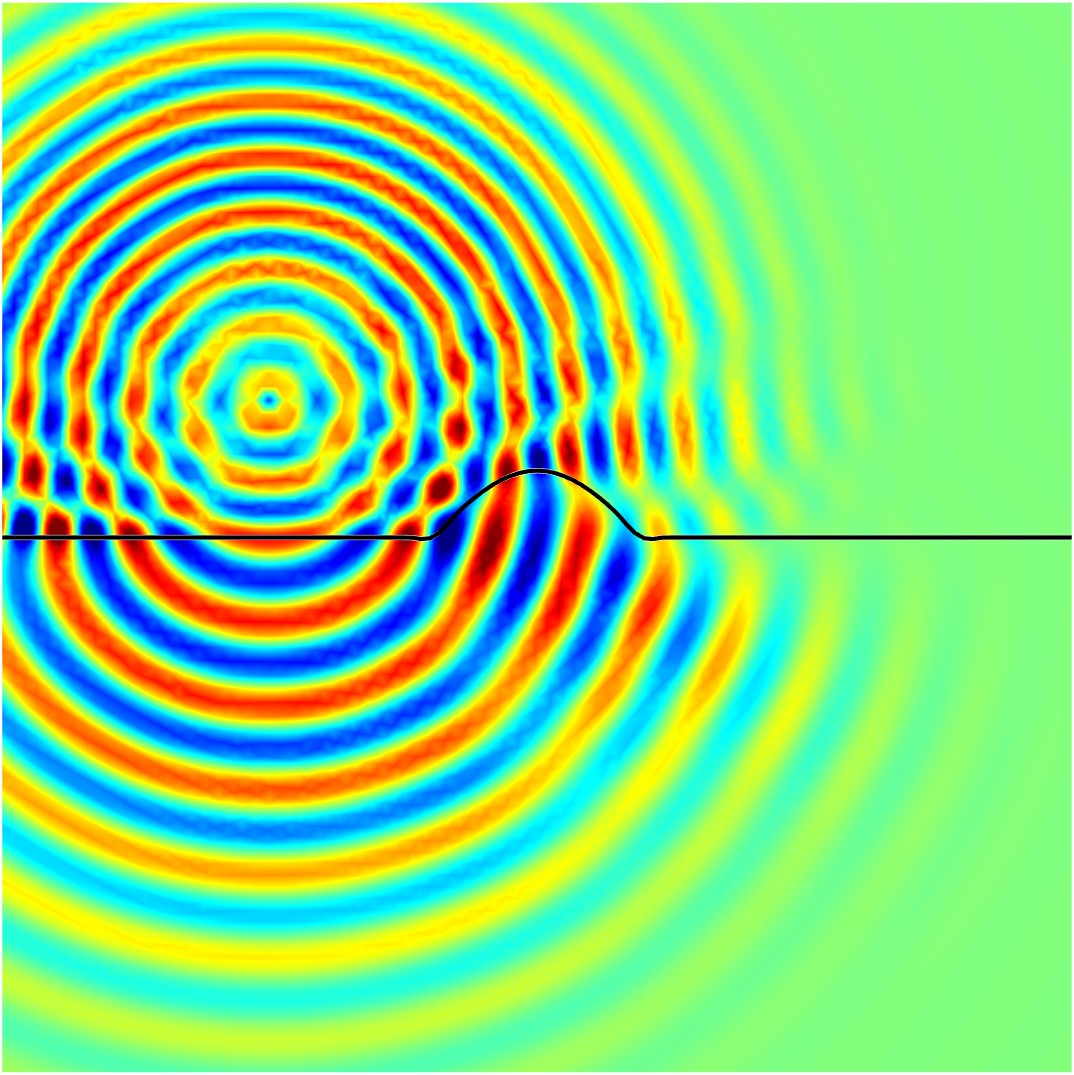}
    }
    \subfloat[]{
        \includegraphics[width=0.25\textwidth]{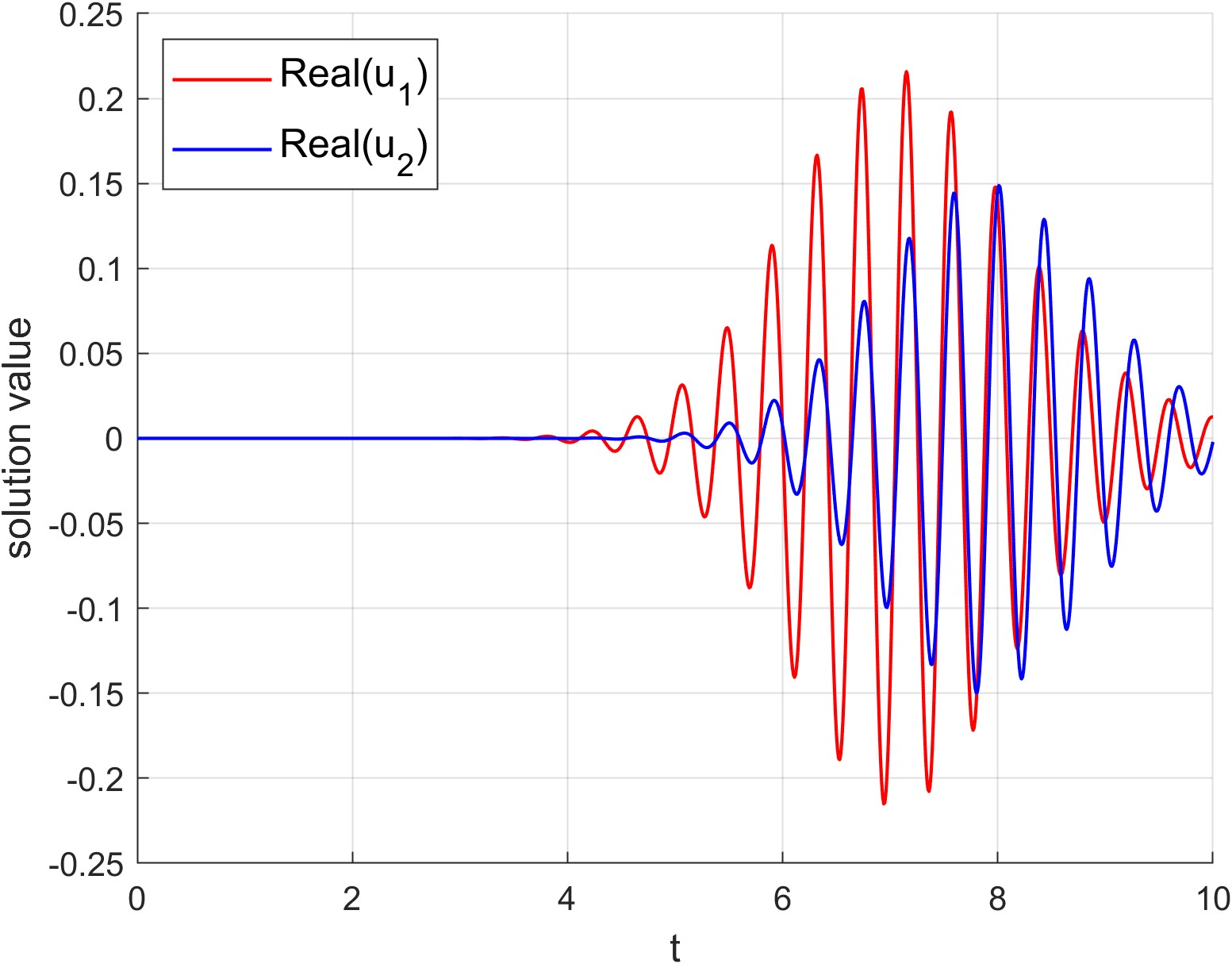}
    }
    \subfloat[]{
        \includegraphics[width=0.25\textwidth]{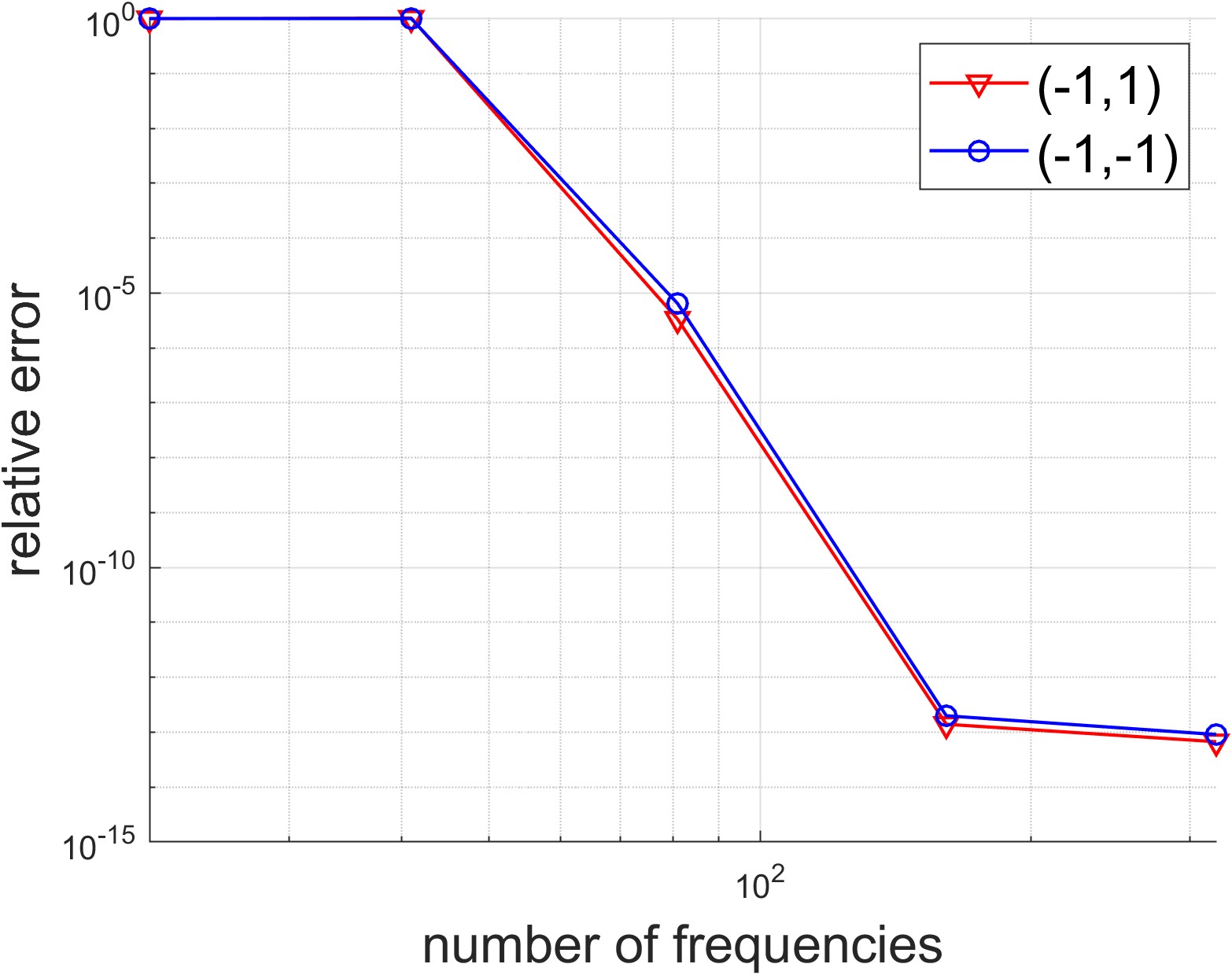}
    }
    \caption{Example 1. Numerical results of the FTH method for the 2-layered medium problem with a point source:
    (a) total fields at $t=6$; (b) time trace of the scattered fields at $\bm x_1=(-1,1)$ and $\bm x_2=(-1,-1)$; and (c) maximum errors $\varepsilon_\infty$ as a function of the number of frequencies.}
    \label{fig:example1-point}
\end{figure}

{\bf Example 1.} (2-layered problems in 2D) First of all, we demonstrate the accuracy of the FTH solver combined with the PML-BIE method for solving the wave equation problems in a 2-layered medium. The interface $\Gamma_1$ is set to be the plane $x_2=0$ locally perturbed by 
$x_2=A\cos(2x_1)\eta(x_1,b_0,b_1)$, where
\begin{equation*}
    \eta(z,b_0,b_1)=
\begin{cases}
1, & |z|\leq b_0,\\
\exp
\left(
\frac{2e^{-1/u}}{u-1}
\right),
& b_0<|z|<b_1,u=\frac{|z|-b_0}{b_1-b_0},\\
0, & |z|\geq b_1,
\end{cases}
\end{equation*}
with $A=0.5,b_0=0.5$and $b_1=1$, see Figure~\ref{fig:example1-plane}(a). We choose the wave speeds in the domains ${\Omega}_1$ and ${\Omega}_2$ as $c_1=1$ and $c_2=1.5$, respectively. Consider the incidence of a Gaussian-pulse plane wave with incident angle $\theta^{\text{inc}}=\pi/4$ or a point source located at $\bm z_0=(-2,1)$ specified in Section~\ref{sec:4.2} and choose $t_0=6$ and $\omega_0 = 15$. We truncate the frequency integration interval to $\omega\in [5, 25]$, outside of which the spectral amplitude is below machine precision. For the plane wave and point source cases, we denote by $(u_1+u_\mathrm{plane}^\mathrm{inc},u_2)$ and $(u_1+u_\mathrm{point}^\mathrm{inc},u_2)$ the pairs of total fields, respectively. Figures~\ref{fig:example1-plane} and \ref{fig:example1-point} display the numerical solutions of the 2-layered wave equation problems (total fields at $t=6$ and time trace of the scattered fields at $\bm x_1=(-1,1)$ and $\bm x_2=(-1,-1)$ and convergence of the maximum errors $\varepsilon_\infty$ at $\bm x_1$ and $\bm x_2$ for $T=10$ with respect to the number of equi-spaced frequencies set for the evaluation of Fourier transform, which clearly verify the fast convergence and high accuracy of the FTH and PML-BIE methods for solving the 2-layered wave equation problems.

\begin{figure}[htbp]
    \centering
    \subfloat[]{
        \includegraphics[width=0.267\linewidth]{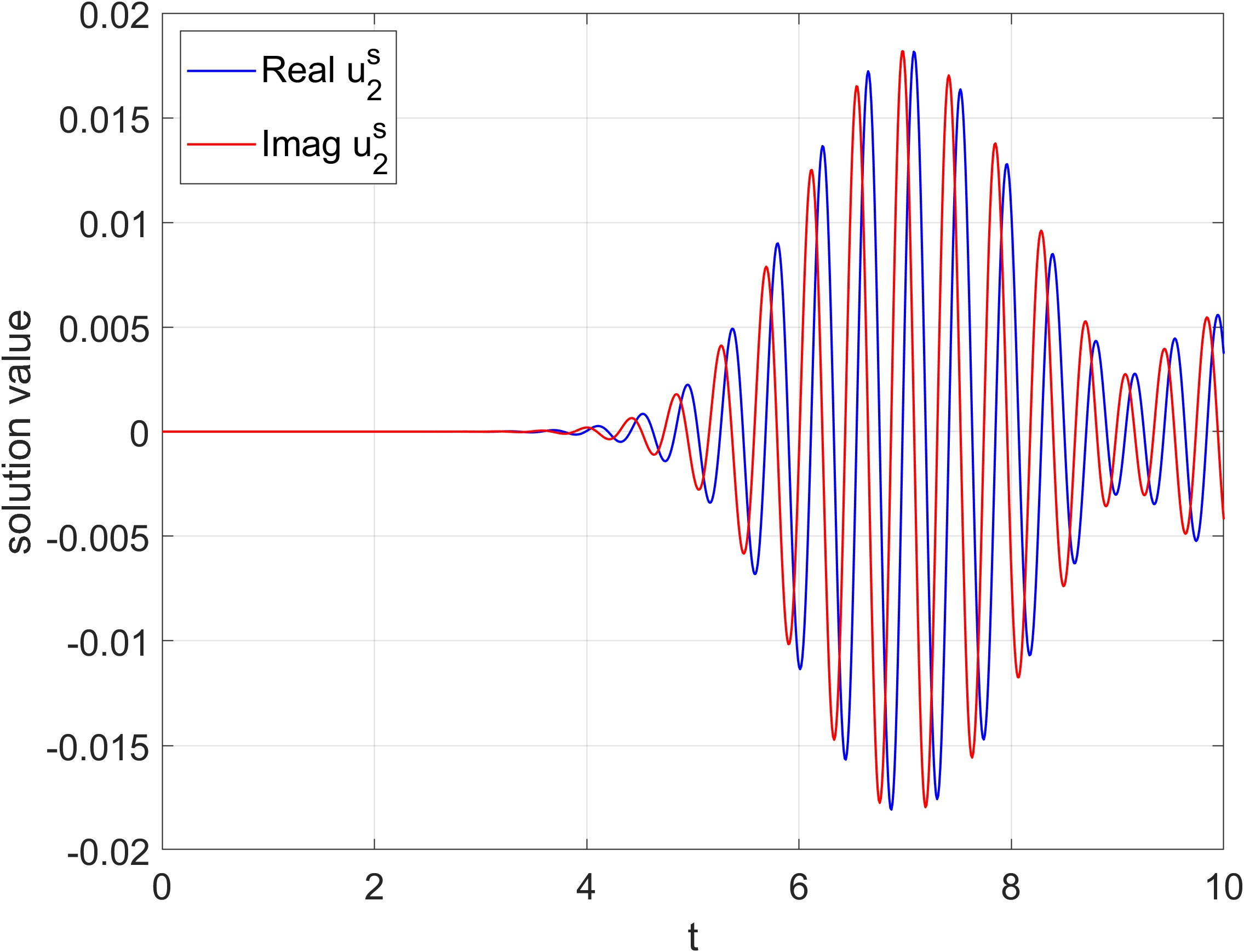}
    }
    \subfloat[]{
        \includegraphics[width=0.267\linewidth]{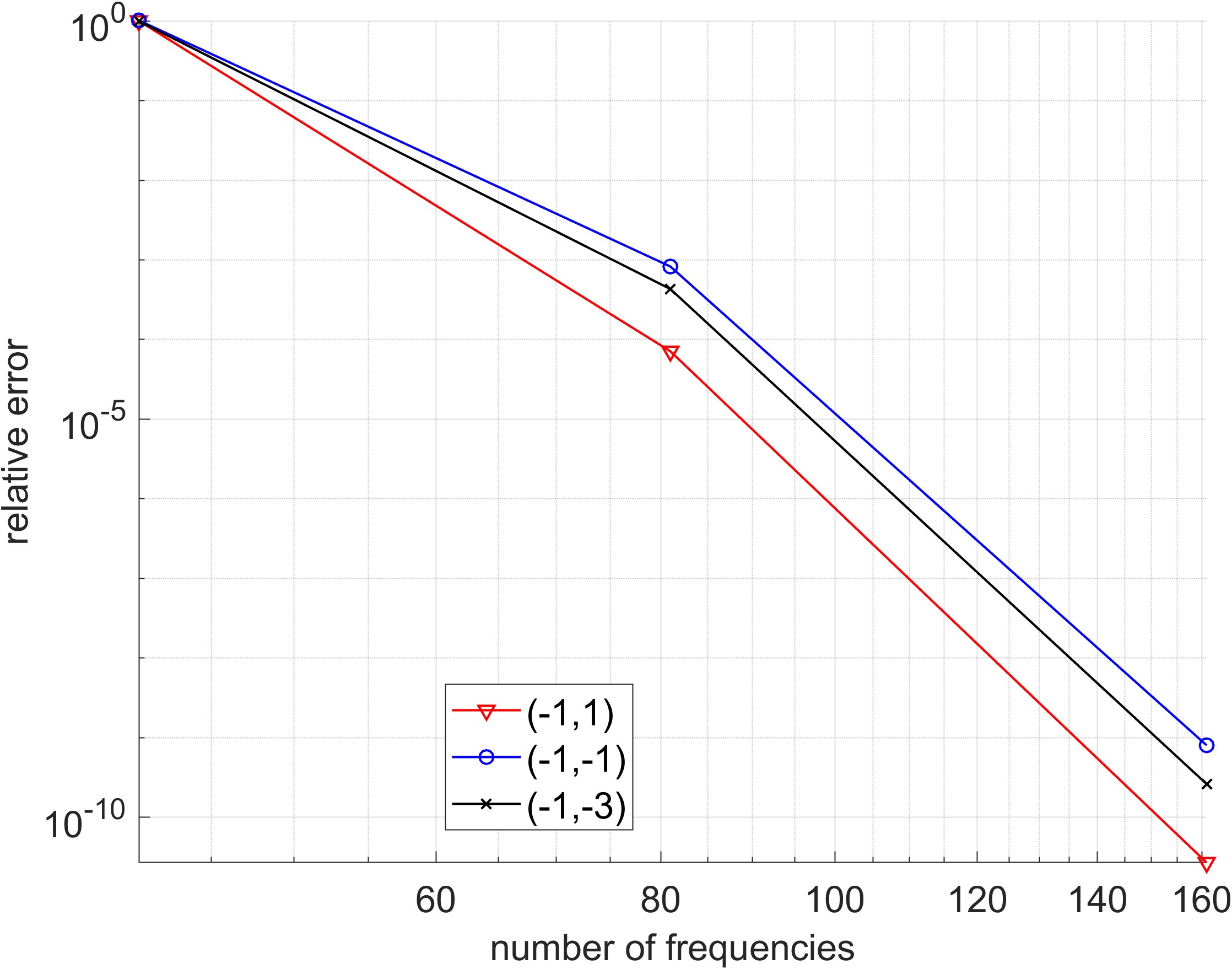}
    }
    \subfloat[]{
        \includegraphics[width=0.267\linewidth]{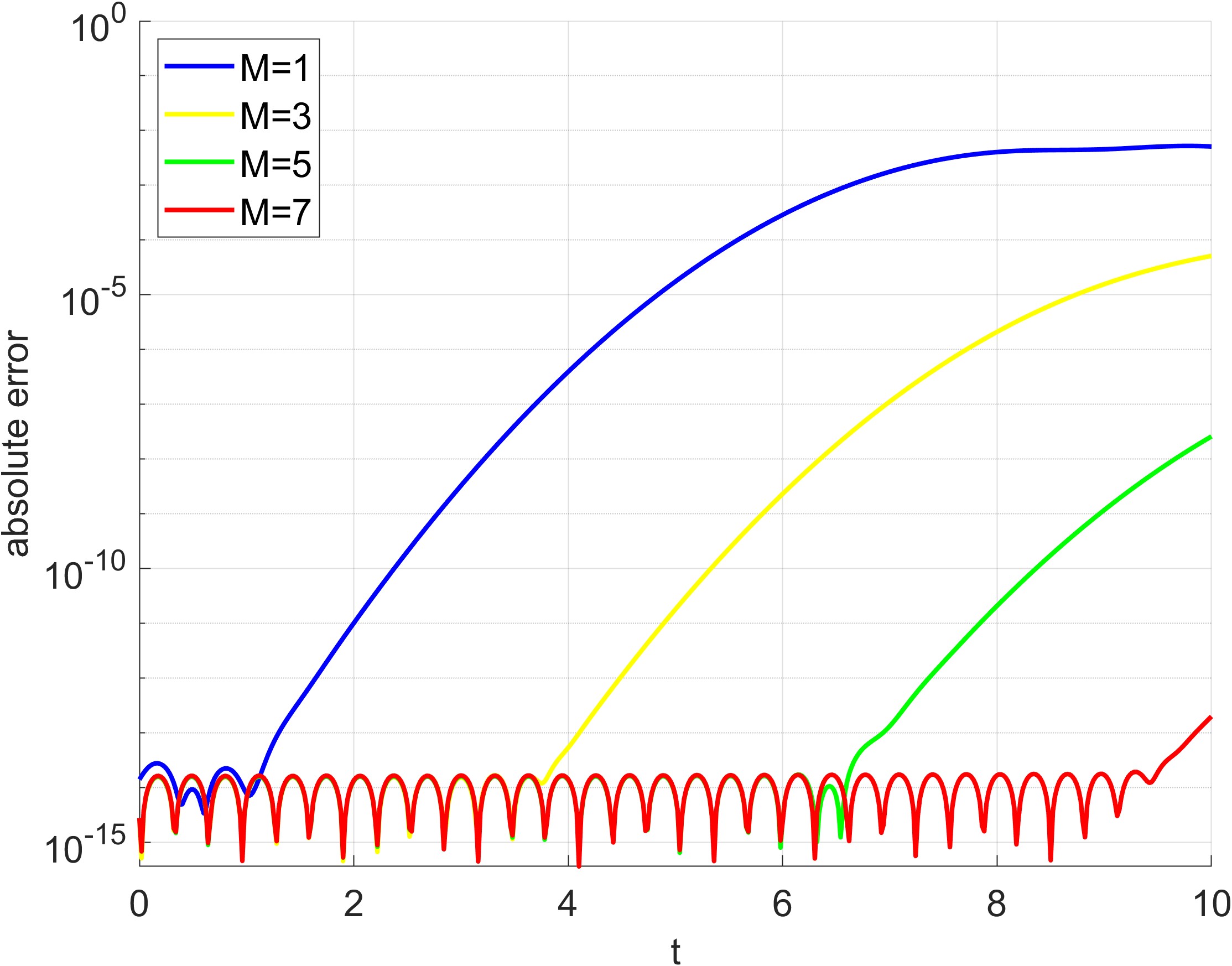}
    }
    \caption{Example 2. Numerical results of the FTH-MS method for the 3-layered medium problem with a point source: (a) time trace of the scattered field $u_2$ at $\bm x=(-1,-1)$; (b) maximum errors $\varepsilon_\infty$ as a function of the number of frequencies; and (c) time trace of the numerical errors $\varepsilon(\bm x,t)$ at $\bm x=(-1,-1)$ for different values of $M$.}
    \label{fig:Example2-3layer-error}
\end{figure}

\begin{figure}[htbp]
    \centering
    \includegraphics[width=0.2\linewidth]{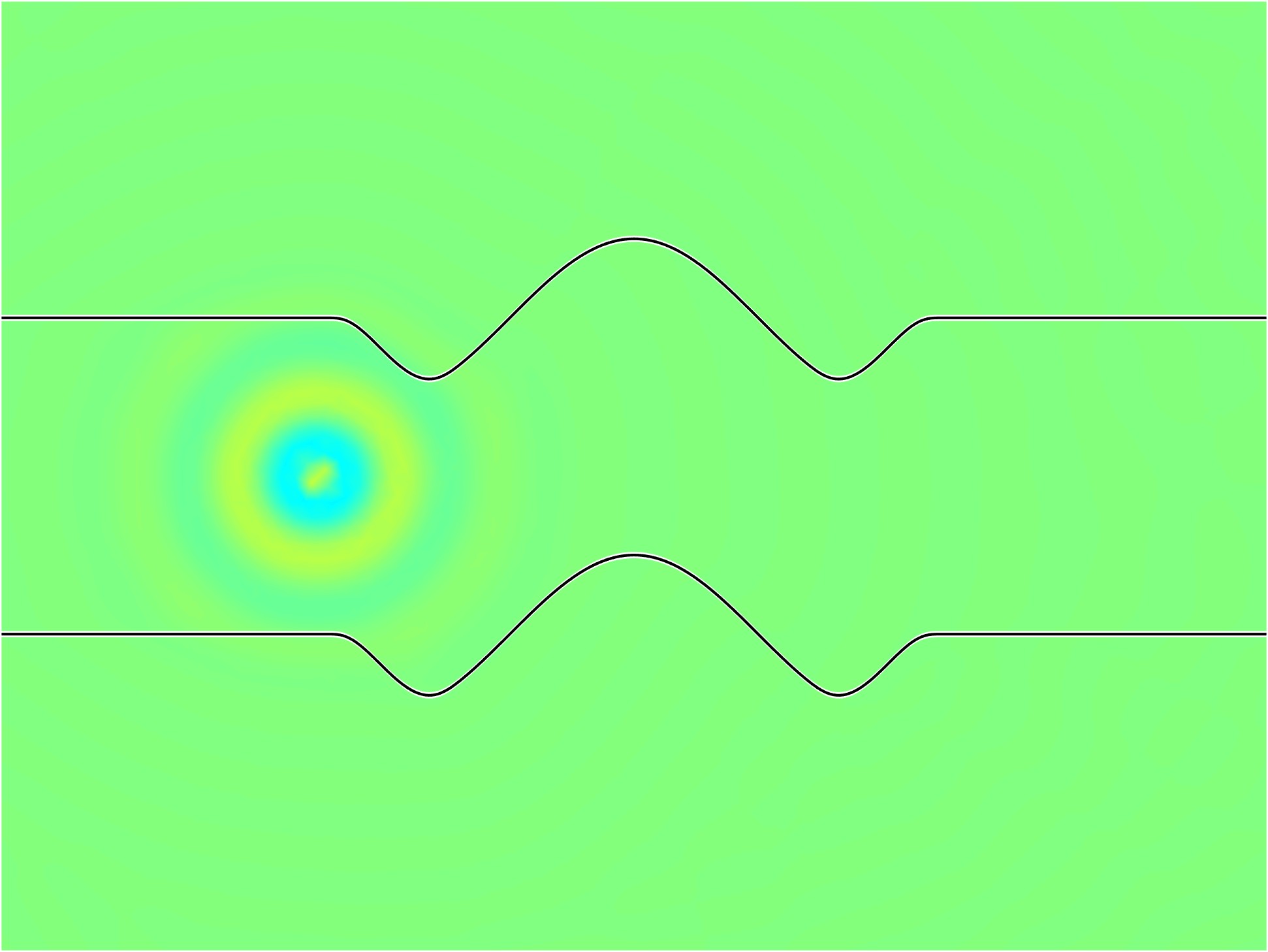}
    \includegraphics[width=0.2\linewidth]{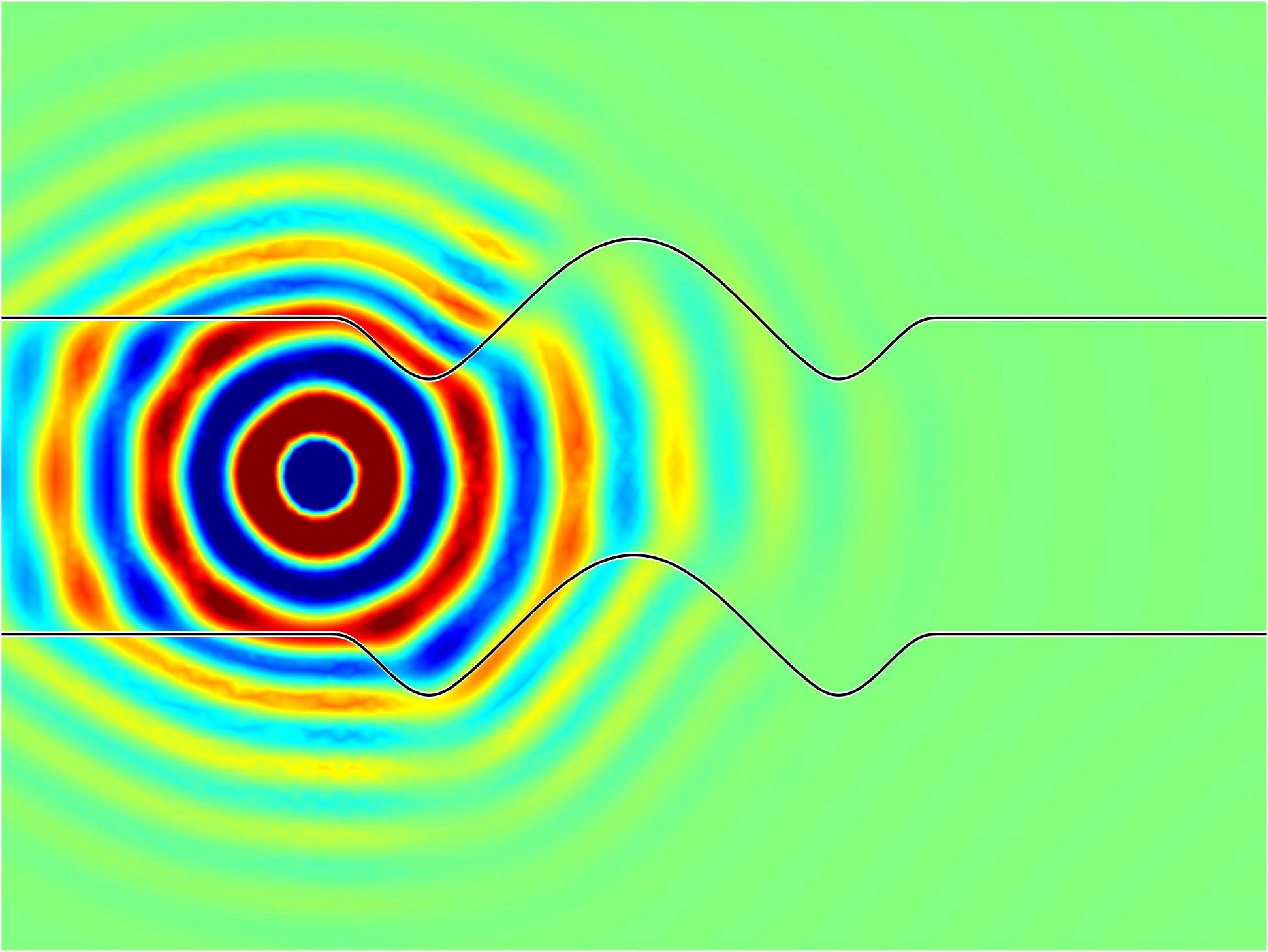}
    \includegraphics[width=0.2\linewidth]{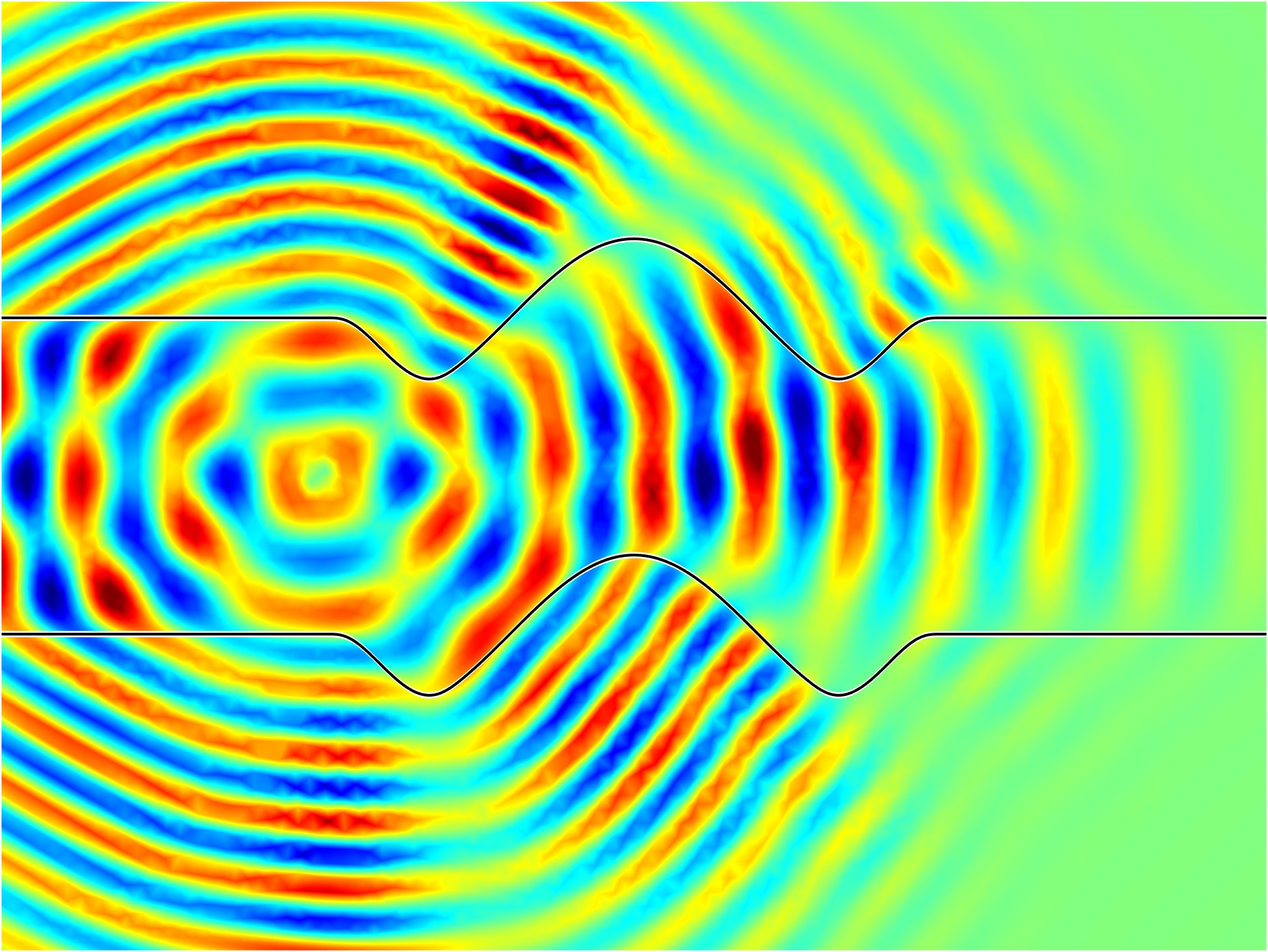}
    \includegraphics[width=0.2\linewidth]{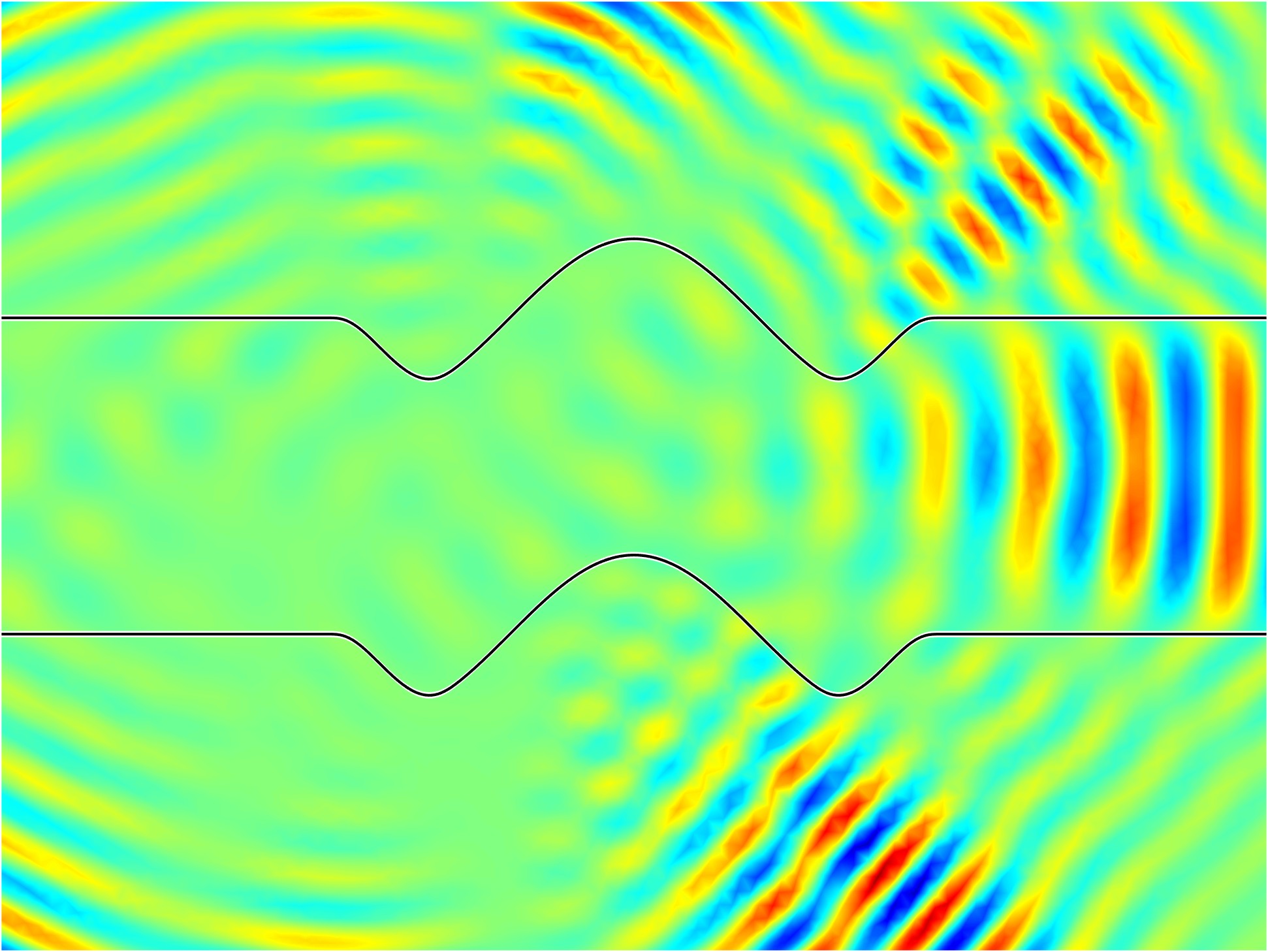}
    \caption{Example 2. Real parts of the total fields at $t=4,6,8,10$ (from left to right) resulted from the FTH-MS method for the scattering of a point source in a 3-layered medium.}
    \label{fig:Example2-3layer-sol}
\end{figure}

\begin{figure}[htbp]
    \centering
    \subfloat[]{
        \includegraphics[width=0.3\linewidth]{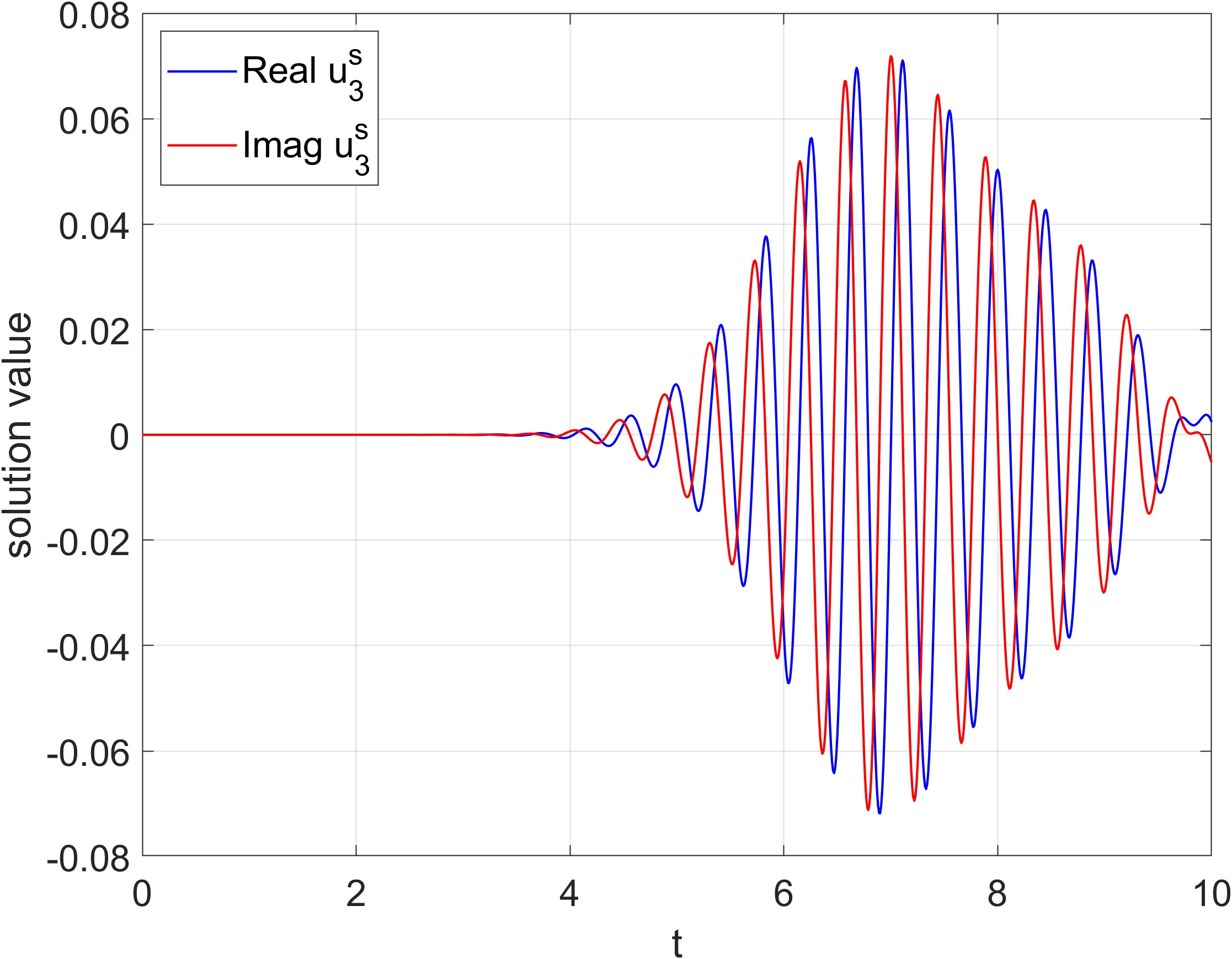}
    }
    \subfloat[]{
        \includegraphics[width=0.3\linewidth]{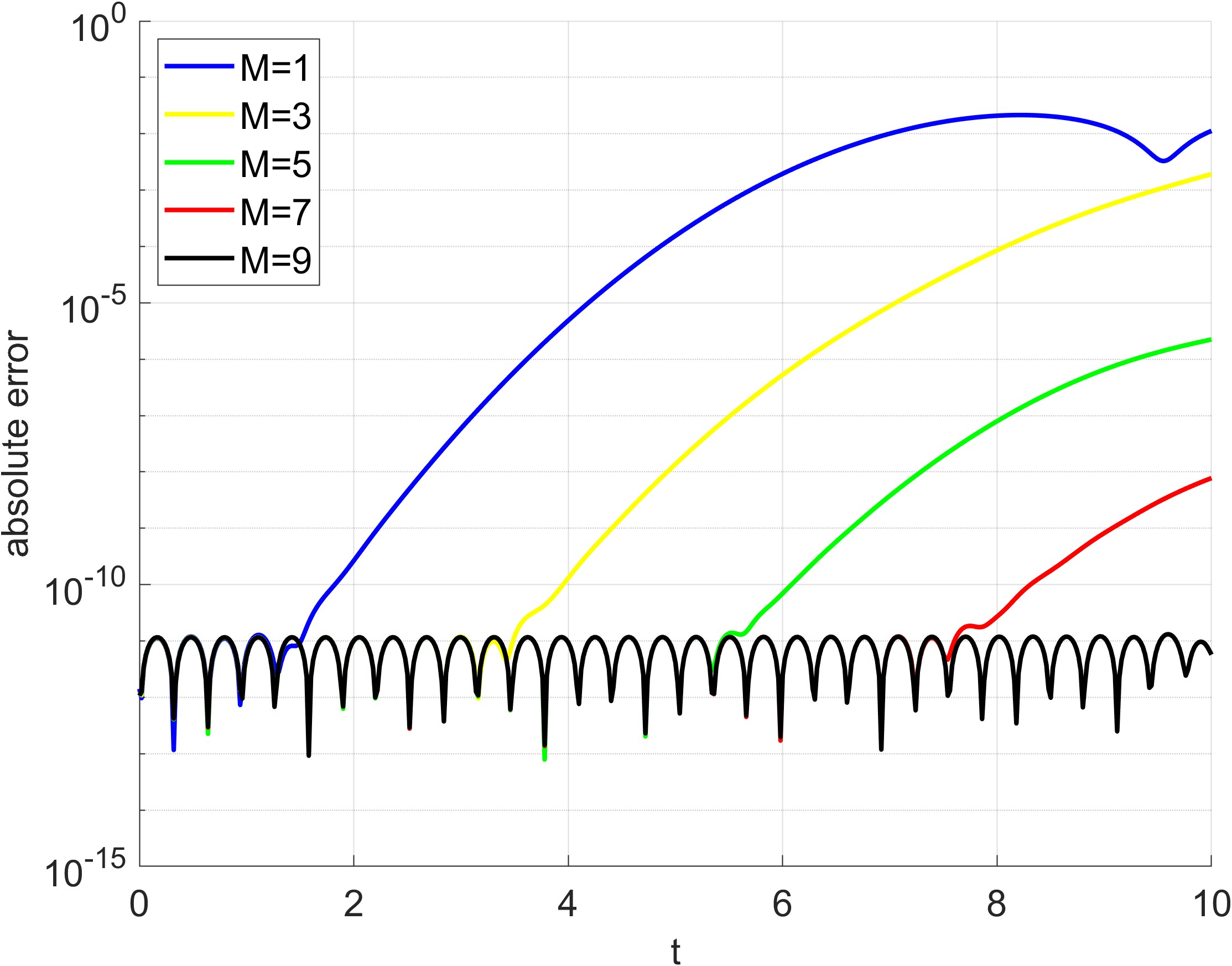}
    }
    \caption{Example 2. Numerical results of the FTH-MS method for the 5-layered medium problem with a point source: (a) time trace of the scattered field $u_3$ at $\bm x=(-1,-3)$; (b) time trace of the numerical errors $\varepsilon(\bm x,t)$ at $\bm x=(-1,-3)$ for different values of $M$.}
    \label{fig:Example2-5layer-error}
\end{figure}

\begin{figure}[htbp]
    \centering
    \includegraphics[width=0.2\linewidth]{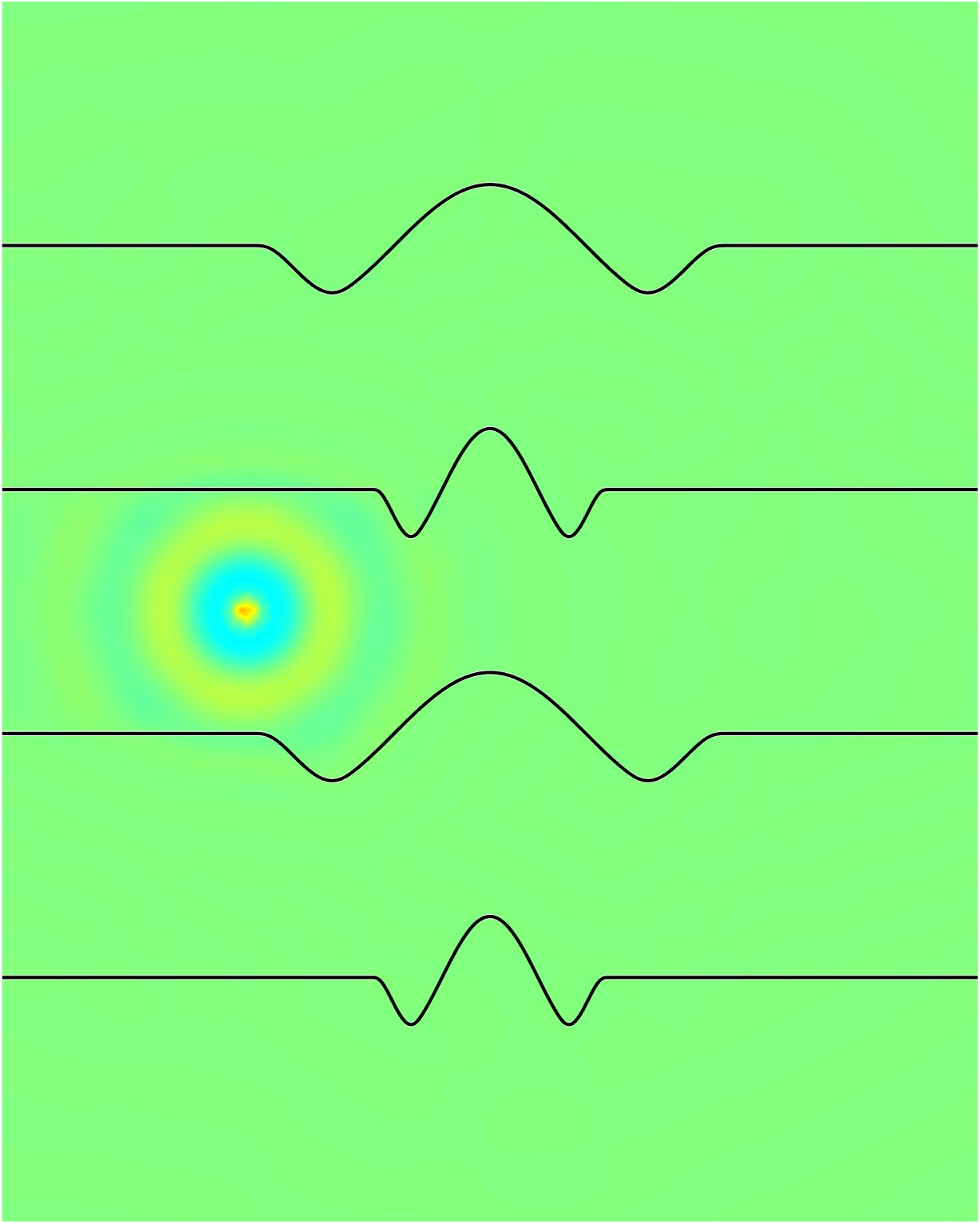}
    \includegraphics[width=0.2\linewidth]{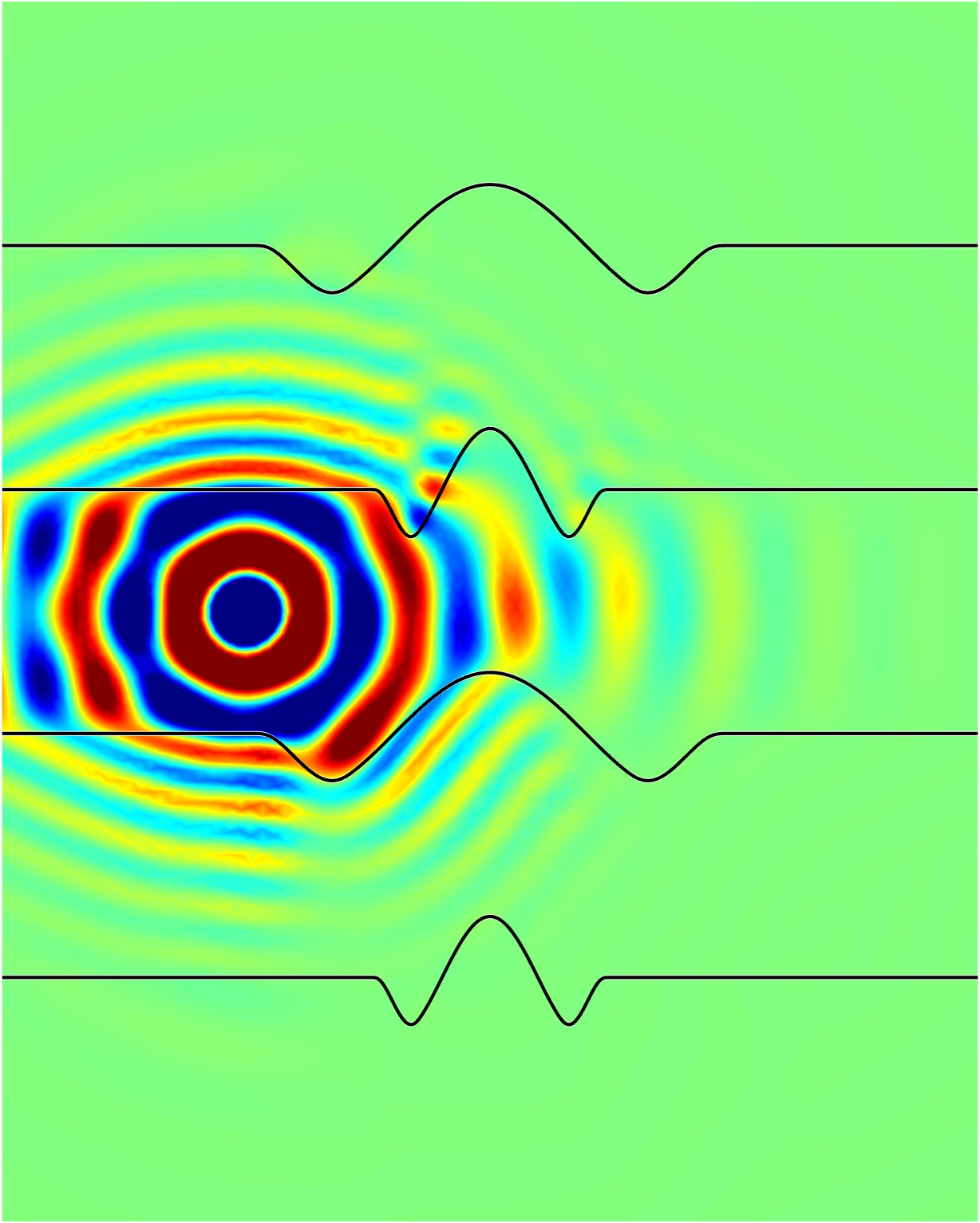}
    \includegraphics[width=0.2\linewidth]{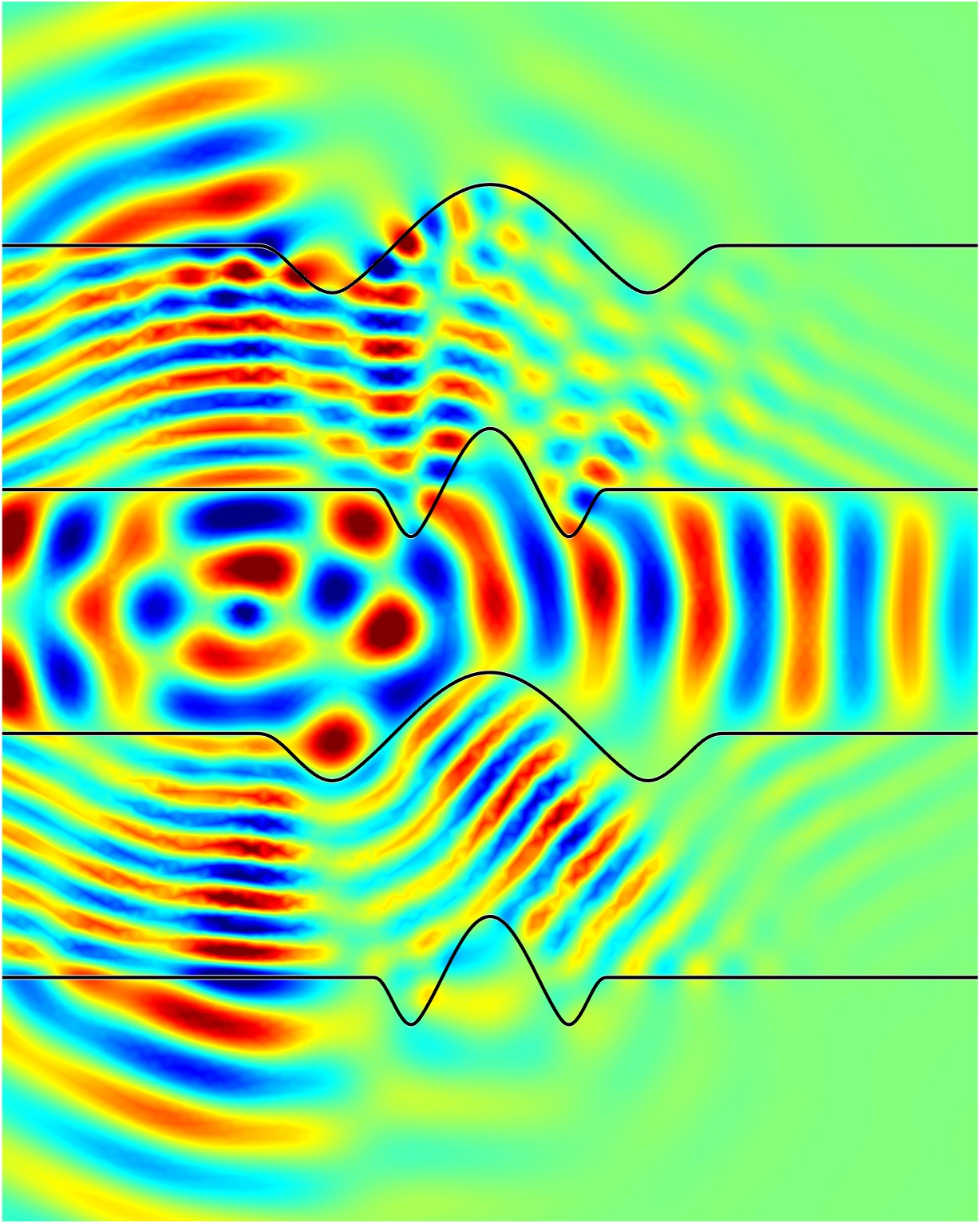}
    \includegraphics[width=0.2\linewidth]{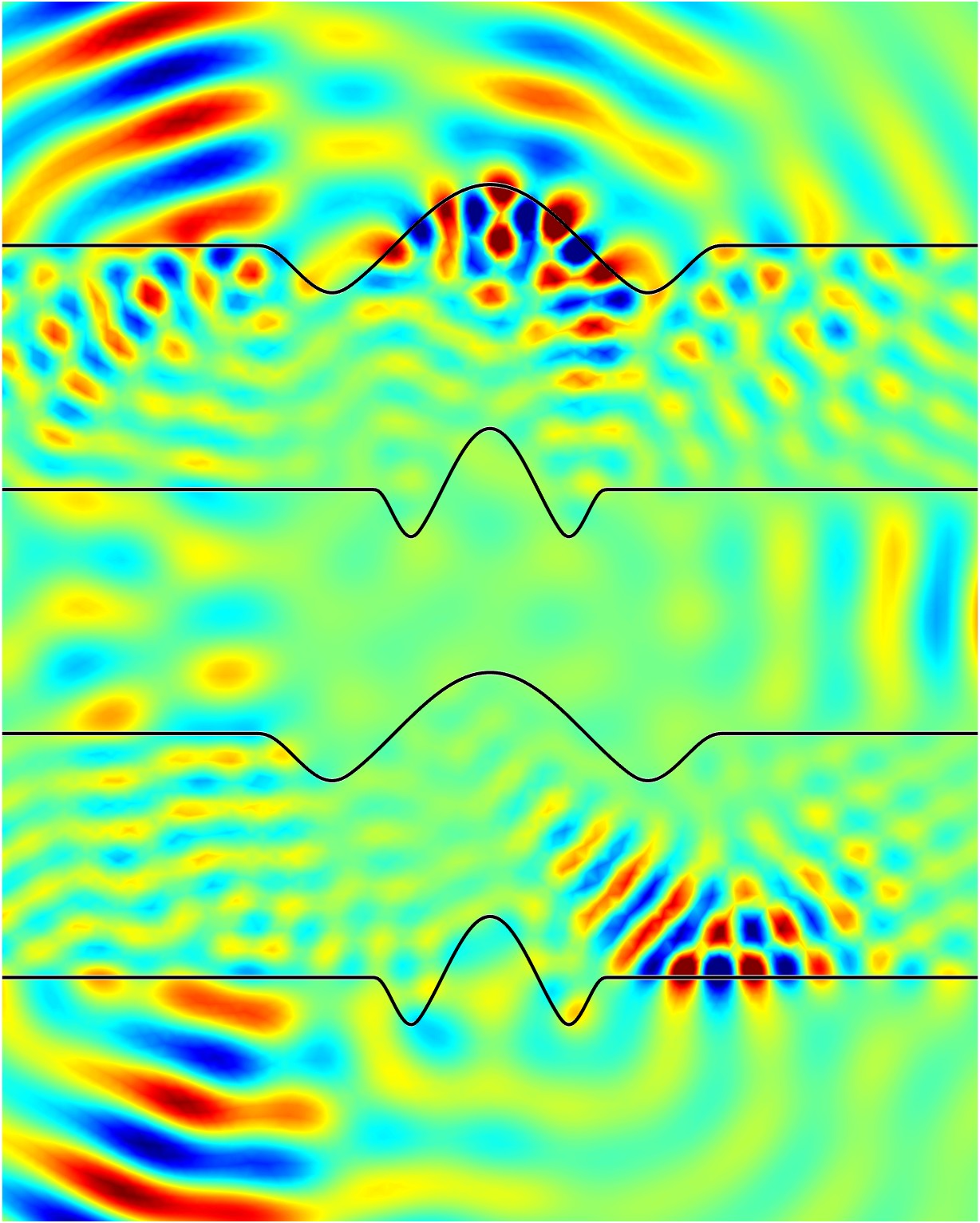}
    \caption{Example 2. Real parts of the total fields at $t=4,6,8,10$ (from left to right) resulted from the FTH-MS method for the scattering of a point source in a 5-layered medium.}
    \label{fig:Example2-5layer-sol}
\end{figure}

{\bf Example 2.} (3\&5-layered problems in 2D) In this example, we study the performance of the proposed FTH-MS solver for the general $N$-layered medium problems with $N=3$ or $N=5$, see Figures~\ref{fig:Example2-3layer-sol} and \ref{fig:Example2-5layer-sol} for the geometric settings. When $N=3$, we set the wave speeds $c_1 = 1$, $c_2 = 1.5$, and $c_3 = 1$. Consider the incidence of a point source located at $\bm z_0=(0,2)$ with the same parameters as in Example 1. The interface $\Gamma_i$ is set to be the plane $x_2=-2(i-1)$ with local perturbation $x_2=0.5\cos(2x_1)\eta(x_1,1,2)-2(i-1)$ for $i=1,2$.
Figure~\ref{fig:Example2-3layer-error}(a,b) presents the time trace of the scattered field $u_2$ at $x=(-1,-1)$ and convergence of the maximum errors $\varepsilon_\infty(\bm x_j)$ at $\bm x_1=(-1,1)$, $\bm x_2=(-1,-1)$ and $\bm x_3=(-1,-3)$ for $T=10$ with respect to the number of equi-spaced frequencies. We also plot in Figure~\ref{fig:Example2-3layer-error}(c) the time trace of the numerical errors $\varepsilon(\bm x,t)$ at $\bm x=(-1,-1)$ for different values of $M$, i.e, the number of multiple-scattering iterations to construct the multiple scattering solutions $\widetilde{u}_j^M$ defined in (\ref{eq:Nlayer-ms-add}). The results show in Figure~\ref{fig:Example2-3layer-error}(b,c) demonstrate the fast convergence and high accuracy of the proposed FTH-MS solver and clearly verify the equivalence result proved in Theorem~\ref{thm:equivalence} which indicates that for a fixed terminal time $T>0$, only a finite number of multiple-scattering iterations is required to get a highly accurate numerical approximation of the solutions to multi-layered medium wave equation problems for all $t\in[0,T]$. The numerical total fields at $t=3,5,7,9$ are shown in Figure~\ref{fig:Example2-3layer-sol}. 

Next, we consider the 5-layered medium problem and test the efficiency of the proposed FTH-MS solver. The wave speeds are set to be $c_i = 1+(j\mod 2)$ for $j=1,2,\cdots,5$. We set the interface $\Gamma_i=\{(x_1,x_2):x_1\in \mathbb{R},x_2=-2(i-1)+0.5\cos(2x_1)\eta(x_1,1,2)\}$ for $i=1,3$ and $\Gamma_i=\{(x_1,x_2):x_1\in \mathbb{R},x_2=-2(i-1)+0.5\cos(2x_1)\eta(2x_1,1,2)\}$ for $i=2,4$. The incident field is generated by a point source, defined as previously, located at $\bm z_0=(-2, -3)$. Figure~\ref{fig:Example2-5layer-error} shows the solution trace at $\bm x=(-1,-3)$ and the relative error of multiple scattering series $\widetilde u_3^{M}$ for different value of $M$ which also verify the high accuracy of the proposed FTH-MS solver. The numerical total fields at various times are presented in Figure~\ref{fig:Example2-5layer-sol}. 

\begin{figure}[htbp]
    \centering
    \includegraphics[width=0.6\linewidth]{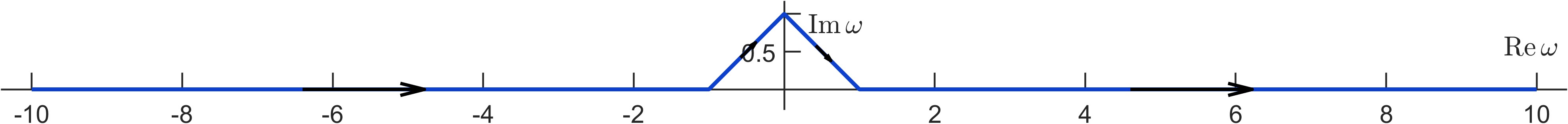}
    \caption{Example 3. Deformed  integration path of Fourier integral.}
    \label{fig:fourier_contour}
\end{figure}

\begin{figure}[htbp]
    \centering
    \subfloat[]{
        \includegraphics[width=0.3\linewidth]{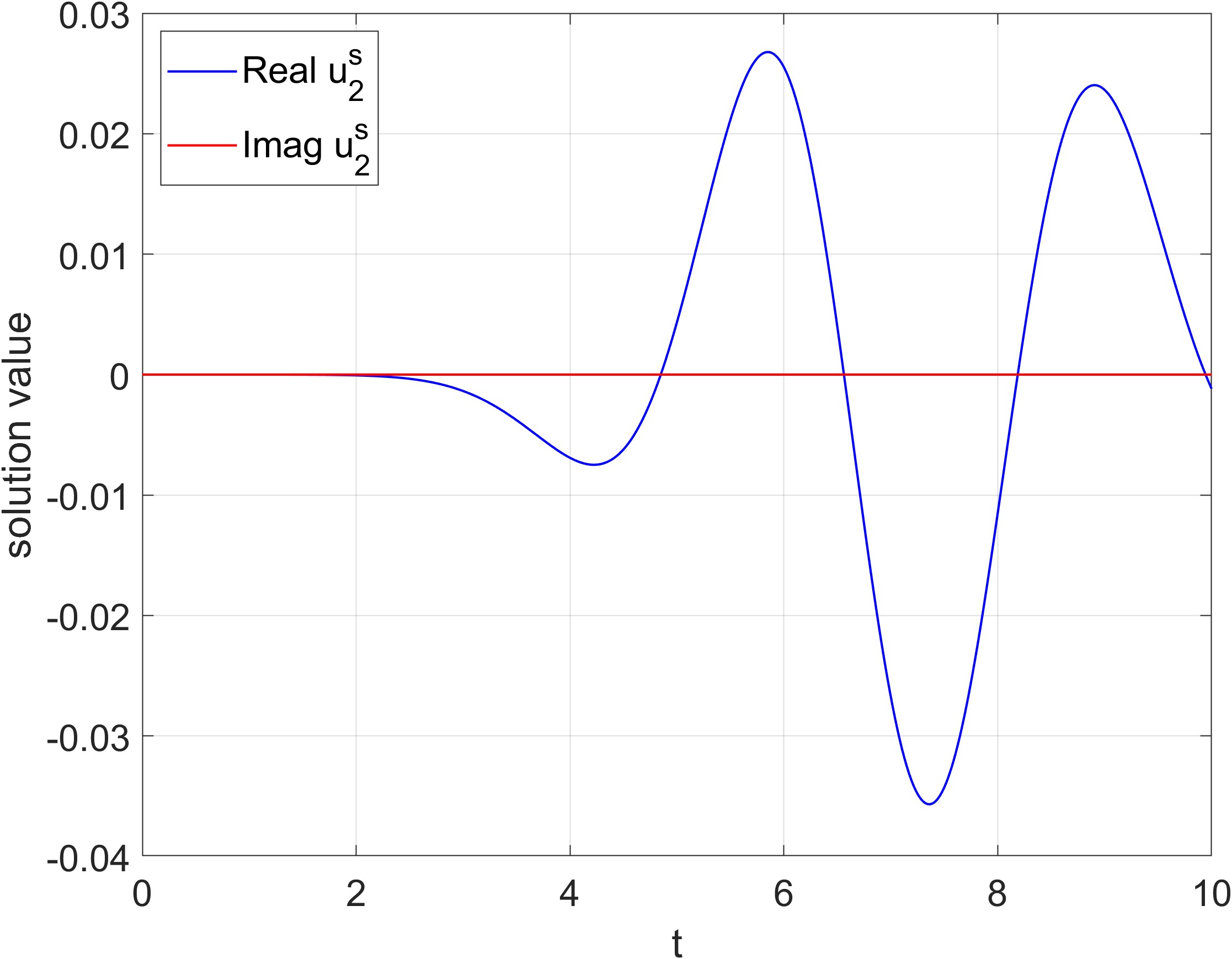}
    }
    \subfloat[]{
        \includegraphics[width=0.3\linewidth]{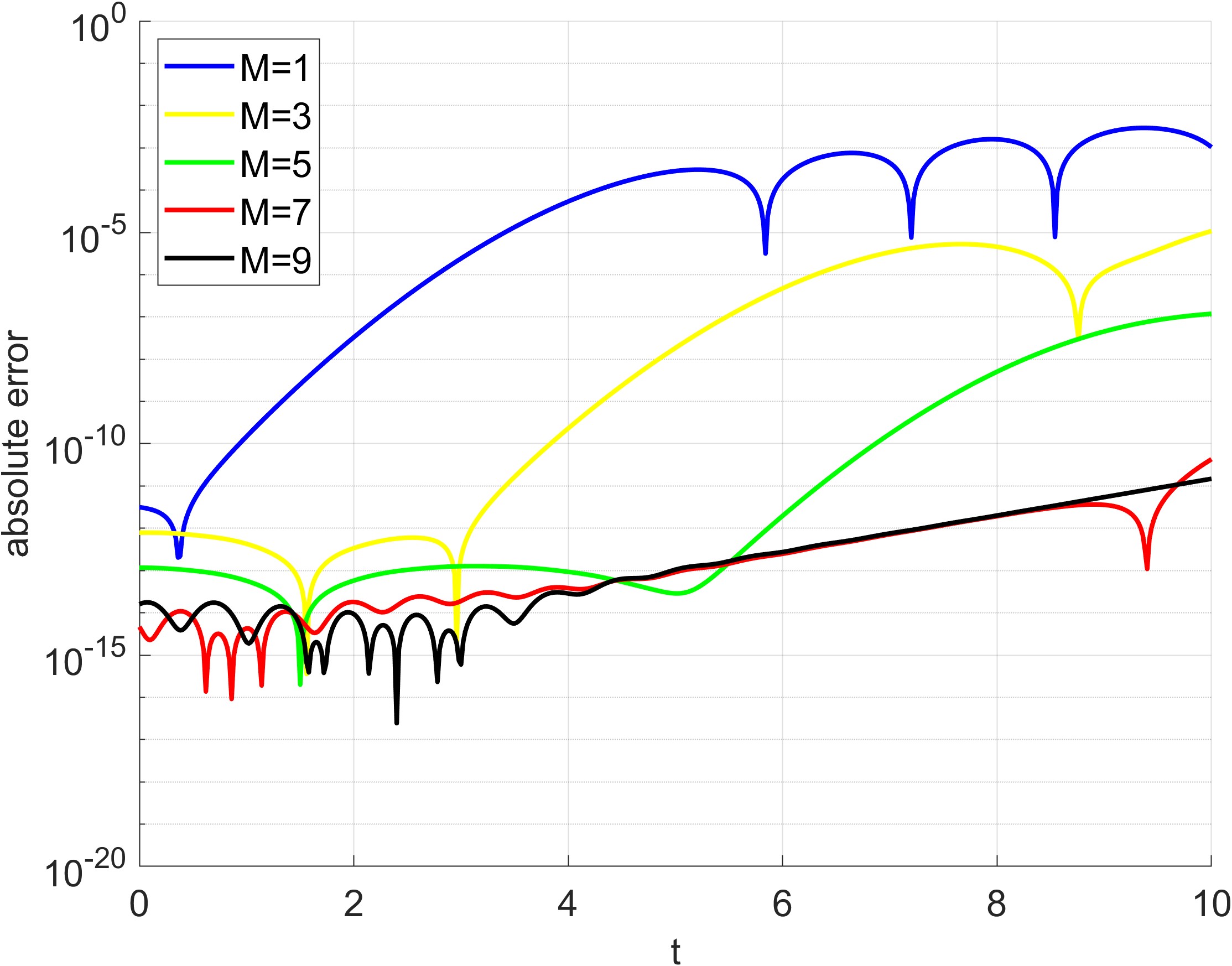}
    }
    \caption{Example 3. Numerical results of the FTH-MS method for the 3-layered medium problem with a point source ($\omega_0=0$): (a) time trace of the scattered field $u_2$ at $\bm x=(-1,-1)$; (b) time trace of the numerical errors $\varepsilon(\bm x,t)$ at $\bm x=(-1,-1)$ for different values of $M$.}
    \label{fig:Example2-3layer-0-error}
\end{figure}

{\bf Example 3.} (3-layered problem in 2D \& low-frequency case) Next, we consider the scattering of a general plane wave which contains low-frequency information in a 3-layered medium in two-dimensions. We set $t_0=5$, $\theta^\mathrm{inc}=\pi/3$ and $\omega_0=0$ which indicates that the solution information is concentrated on a wide frequency interval $[-W,W]$ which implies that both low-frequency and regular-frequency problems should be treated. As aforementioned in Remark~\ref{rem:low-freq}, one way to avoid the destroyed accuracy of PML-BIE method for low-frequency case is to deform the integral for the inverse Fourier transform above the real axis~\cite{chew1999waves}. As shown in Figure~\ref{fig:fourier_contour}, we choose a special and new integral path of the frequencies to calculate the inverse Fourier transform $\mathbb{F}^{-1}(\widetilde U_{j,m}^+)$ and $\mathbb{F}^{-1}(\widetilde U_{j+1,m}^-)$. 
Figure \ref{fig:Example2-3layer-0-error} shows the solution trace at $\bm x=(-1,-1)$ and the relative error of multiple scattering series $\widetilde u_2^{M}$  for different value of $M$. As shown in Figure \ref{fig:Example2-3layer-0-error}(b), an increase of the numerical error $\varepsilon$ as $\sim\mathcal{O}(e^t)$ happens due to the fact that the new integral path will bring an exponential growth factor in the inverse Fourier transform depending on the largest distance of the new integral path to zero, as predicted in Remark \ref{rem:low-freq}.

\begin{figure}[htbp]
    \centering
    \includegraphics[width=0.2\linewidth]{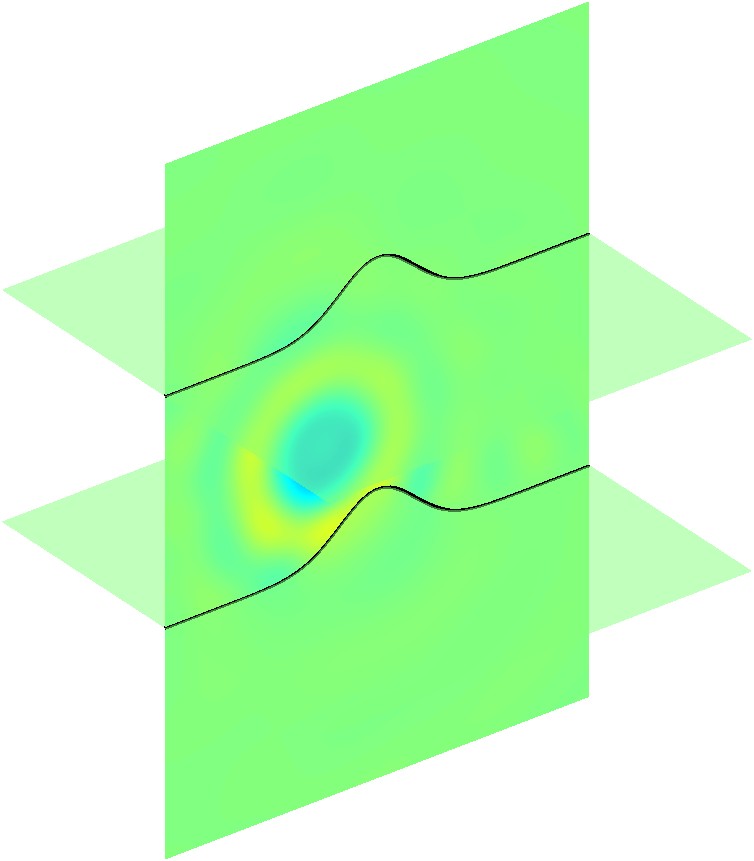}
    \includegraphics[width=0.2\linewidth]{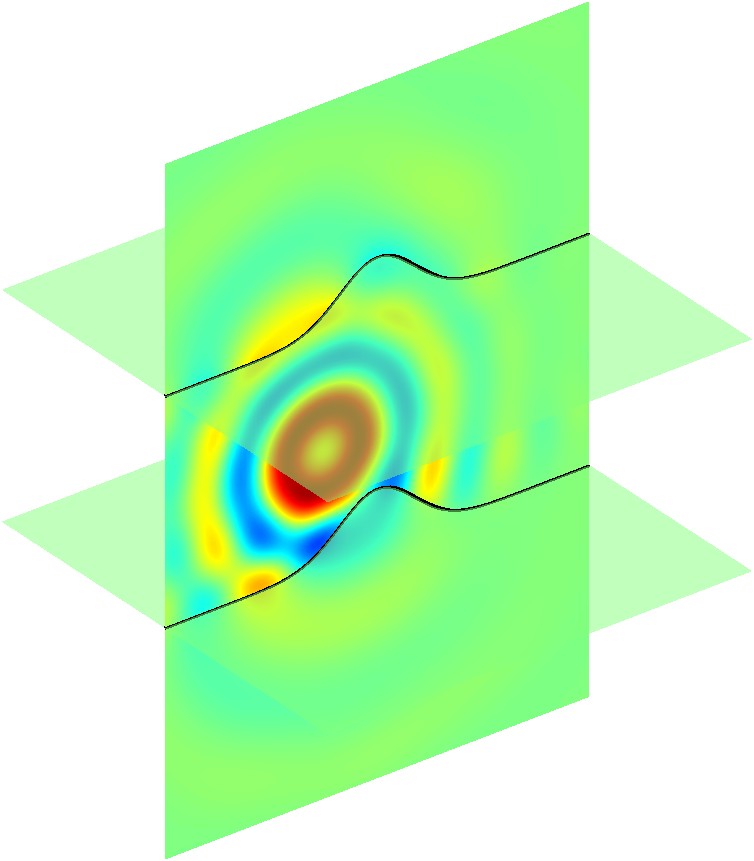}
    \includegraphics[width=0.2\linewidth]{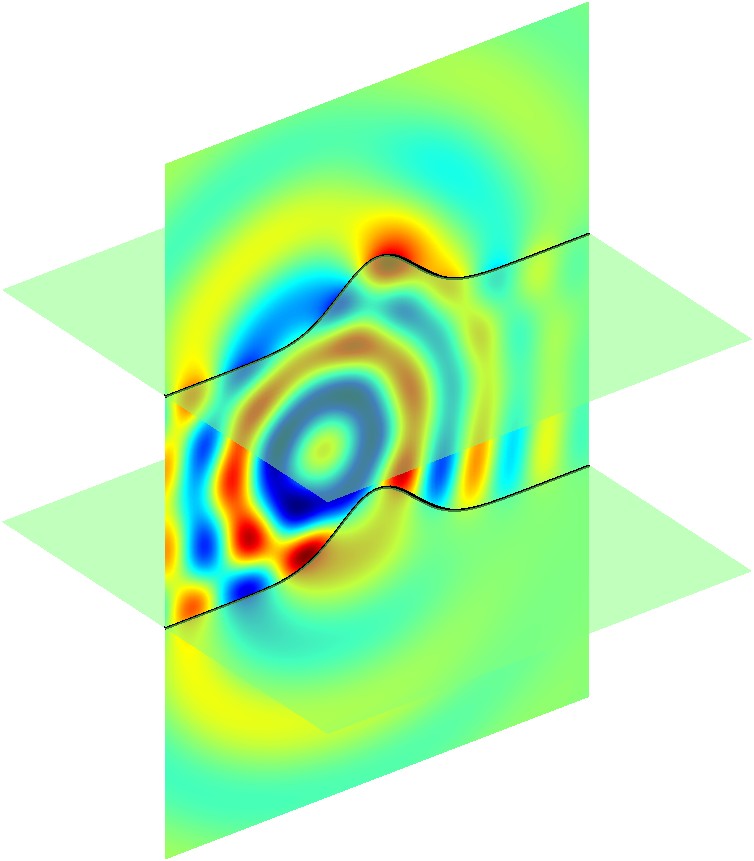}
    \includegraphics[width=0.2\linewidth]{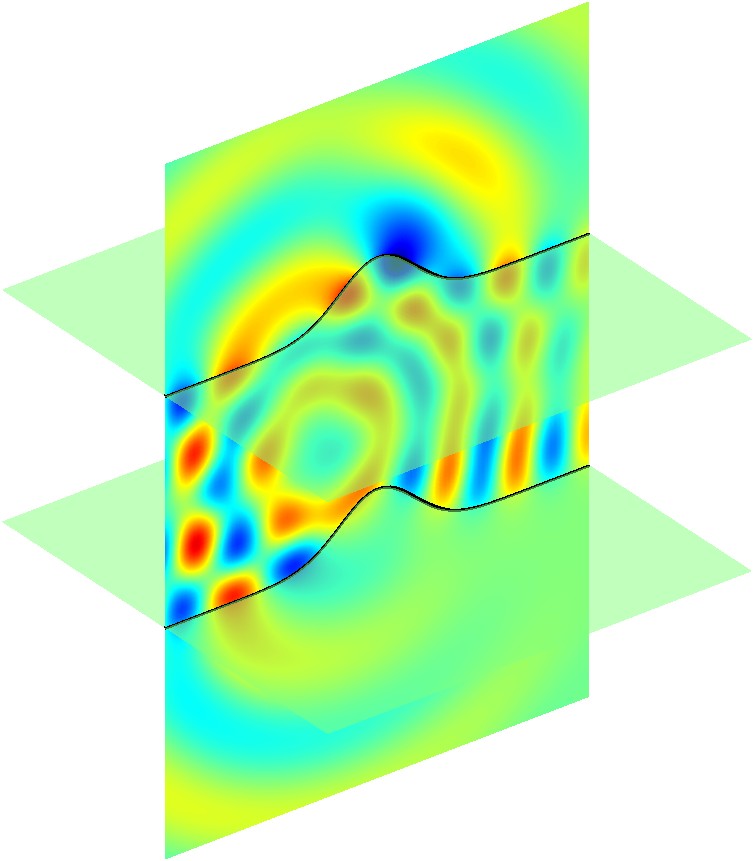}
    \caption{Example 4. Real parts of total fields at $t=4,5,6,7$ (from left to right) resulted from the FTH-MS method for the scattering of a point source in a 3-layered medium in three dimensions.}
    \label{fig:example4-3layer-sol}
\end{figure}

{\bf Example 4.} (3-layered problem in 3D) In this last example, we study the application of the proposed FTH-MS solver for the more challenging three-dimensional layered-medium wave equation problem with $N=3$. The incident field is chose to be a Gauss-pulse point source which is an inverse Fourier transform of a frequency-domain function
\begin{equation}
    \label{eq:incident_field_freq}
U_{\text{point}}^{\text{inc}}(\bm x,\omega,\bm z_0)(\bm x, \omega) = e^{-\frac{(\omega-\omega_0)^2}{2}} e^{i\omega t_0} \frac{e^{i\kappa_2|\bm x-\bm z_0|}}{|\bm x-\bm z_0|}.
\end{equation}
where $t_0=5$, $\omega_0=10$ and $\bm z_0=(-0.5,0.5,-1)$. For simplicity, we set the interfaces $\Gamma_j,j=1,2$ to be 
\begin{equation*}
\Gamma_j=\{\bm x\in\R^3: x_3=0.5e^{-4(x_1^2+x_2^2)}-2(j-1)\},\quad j=1,2.
\end{equation*}
which can be regarded as a local perturbation of the flat planes. The wave speeds in each layer is given by $c_j= 1+(j\mod 2)$ for $j=1,2,3$. The PML-BIE equation method is utilized to solve the two-layered medium subproblems for 51 equi-spaced frequencies for $\omega\in [1,19]$. Numerical solutions at various times resulted from the FTH-MS solver with multiple-scattering iterations $M=5$ is shown in Figure~\ref{fig:example4-3layer-sol}. Further acceleration for the three-dimensional problems will be left for future work.

\section{Conclusion and further extensions}
\label{sec:6}

This paper proposes a novel idea of decomposing the wave equation problem in a complicated multi-layered medium as a multiple-scattering series of a sequence of multiple wave equation problems in two-layered media, which avoids the construction of a global discretized system for the entire multi-layered structure, and then develops a highly accurate FTH-MS solver based on the Fourier transform and PML-BIE method. Both multiplicative-
and additive-type strategies are developed and equivalence between the original solution and the re-modeled multiple-scattering series has been established in light of the finite propagation speed. Further extensions and analysis of the proposed FTH-MS solver lie in
\begin{itemize}
\item Techniques to treat general incidences consisting of low-frequency information and achieve stable long-time simulation;
\item Efficient way to accelerate the solver for the numerical evaluation of three-dimensional problems, for example the fast multipole method for PML-BIE~\cite{goodwill2025fast};
\item Application to more complicated elastic and electromagnetic wave equation problems in multi-layered media~\cite{BaoEtAl2024Highly} as well as the layered-medium problem with deep cavity~\cite{Lai2014Cavity}.
\end{itemize}


\section*{Acknowledgments} This work is partially supported by the
National Key R\&D Program of China (2024YFA1016000), Beijing Natural 
Science Foundation (JR26005), the Strategic Priority Research Program of the Chinese Academy of Sciences (XDB0640000), and the NSFC grants (12171465 and 12288201).

\small
\bibliographystyle{siam}
\bibliography{refs}

\end{document}